\documentclass[12pt,a4paper,oneside]{amsart}
\usepackage{amssymb,graphicx,color,esint,stmaryrd,hyperref,verbatim}

\usepackage[a4paper]{geometry}
\newcommand{\R}{{\mathbb R}}

\newcommand{\vol}{{\operatorname{vol}}}
\newcommand{\Z}{{\mathcal Z}}
\newcommand{\A}{{\mathcal A}}
\newcommand{\eps}{{\epsilon}}
\newcommand{\K}{{\mathcal K}}
\newcommand{\I}{{\mathcal I}}

\newcommand{\C}{{\mathbb C}}

\newcommand{\tr}{{\textup{tr}}}
\newcommand{\field}{{\mathbb K}}

\newtheorem{thm}{Theorem}

\newtheorem{prop}{Proposition}
\newtheorem{lemma}{Lemma}

\theoremstyle{definition}
\newtheorem{defini}{Definition}

\theoremstyle{remark}
\newtheorem{remark}{Remark}

\title[Beurling-Landau's density]{Beurling-Landau's density on compact manifolds}
\author{Joaquim Ortega-Cerd\`a}
\address{Dept.\ Matem\`atica Aplicada i An\`alisi,
 Universitat  de Barcelona,
Gran Via 585, 08007 Bar\-ce\-lo\-na, Spain}
\email{jortega@ub.edu}
\author{Bharti Pridhnani}
\address{Dept.\ Matem\`atica Aplicada i An\`alisi,
 Universitat  de Barcelona,
Gran Via 585, 08007 Bar\-ce\-lo\-na, Spain}
\email{bharti.pridhnani@ub.edu}
\thanks{Supported by the project MTM2008-05561-C02-01 and the CIRIT grant
2009SGR-1303}
\date{\today}

\begin{document}

\maketitle

\begin{abstract}
Given a compact Riemannian manifold $M$, we consider the subspace of $L^2(M)$ 
generated by the eigenfunctions of the Laplacian of eigenvalue less than $L\geq 1$. 
This space behaves like a space of polynomials and we have an analogy with the 
Paley-Wiener spaces. We study the interpolating and Marcienkiewicz-Zygmund (M-Z)
families and provide necessary conditions for sampling and interpolation in terms 
of the Beurling-Landau densities. As an application, we prove the equidistribuition 
of the Fekete arrays on some compact manifolds.
\end{abstract}

\section*{Introduction}
Let $(M,g)$ be a smooth, connected, compact Riemannian manifold without
boundary, of dimension $m\geq 2$. Let $dV$ and $\Delta_M$ be the volume element 
and the Laplacian on $M$ associated to the metric $g$, respectively. The Laplacian 
is given in local coordinates by
\[
\Delta_{M}f=\frac{1}{\sqrt{|g|}}\sum_{i,j}\frac{\partial}{\partial
x_{i}}\left(\sqrt{|g|}g^{ij}\frac{\partial f}{\partial
x_{j}}\right),
\]
where $|g|=|\textup{det}(g_{ij})|$ and $(g^{ij})_{ij}$ is the inverse matrix
of $(g_{ij})_{ij}$. Since $M$ is compact, $g_{ij}$ and all its
derivatives are bounded and we assume that the metric
$g$ is non-singular at each point of $M$.\\
\\
By the compactness of $M$, the spectrum of the Laplacian is discrete
and there is a sequence of eigenvalues
\[
0\leq \lambda^2_{1}\leq\lambda^2_{2}\leq\cdots \to \infty
\]
and an orthonormal basis $\phi_{i}$ of smooth real eigenfunctions of the
Laplacian i.e. $\Delta_{M}\phi_{i}=-\lambda^2_{i}\phi_{i}$. Thus, $L^{2}(M)$ 
decomposes into an orthogonal direct sum of eigenfunctions of the Laplacian.\\
\\
We consider the following subspaces of $L^{2}(M)$.
\[
E_{L}=\left\{f\in L^{2}(M)\ :\
f=\sum^{k_{L}}_{i=1}\beta_{i}\phi_{i},\
\Delta_{M}\phi_{i}=-\lambda^2_{i}\phi_{i},\ \lambda_{k_{L}}\leq
L\right\},
\]
where $L\geq 1$ and $k_{L}=\dim E_{L}$. $E_L$ consists of functions in 
$L^2(M)$ with a restriction on the support of its Fourier transform. It is, in a 
sense, the Paley-Wiener space on $M$ with bandwidth $L$.\\
\\
The goal of this work is to extend the theory of Beurling-Landau on the 
discretization of functions in the Paley-Wiener space on $\R^n$ to functions 
in $M$. This should be possible because there is already a literature on the 
subject in the case $M=\mathbb S^m$ (see \cite{Mar} for more details). 
In the present work, we study the interpolating and Marcinkiewicz-Zygmund 
families for the spaces $E_L$. We prove some basic facts about them 
and give necessary conditions in terms of the 
Beurling-Landau's density. More precisely, our main result is:
\begin{thm}\label{densitythmintro}
Let $\Z$  be a triangular family in $M$. If $\Z$ is an $L^2$-\textup{M-Z} family 
then there exists a uniformly separated $L^2$-\textup{M-Z} family 
$\tilde{\Z}\subset \Z$ such that
\[
D^{-}(\tilde{\Z})\geq 1.
\]
If $\Z$ is an $L^2$-interpolating family then it is uniformly separated and
\[
D^{+}(\Z)\leq 1,
\]
where $D^{+}$ and $D^{-}$ are the upper and lower Beurling-Landau's 
density (see Definition \ref{defBeurlingLandauDensity} for more details), 
respectively,
\end{thm}
In the last section, we study the Fekete families for the spaces $E_L$. Fekete points 
are the points that maximize a Vandermonde-type determinant that appears in the 
polynomial Lagrange interpolation formula. We show their connection with the 
interpolating and M-Z families and prove the asymptotic equidistribution of the Fekete 
points on some compact manifolds. Our main result in this direction is:
\begin{thm}
Let $M$ be an admissible manifold and $\Z=\left\{\Z(L)\right\}_{L\geq 1}$ be any 
array such that $\Z(L)$ is a set of Fekete points of degree $L$. Consider the measures 
$\mu_L=\frac{1}{k_L}\sum^{k_L}_{j=1}\delta_{z_{Lj}}$. Then 
$\mu_L$ converges in the weak-$\ast$ topology to the normalized volume measure on 
$M$.
\end{thm}
For the precise definition of an admissible manifold see Definition \ref{defadmissible}. 
The key idea in proving the above theorem is the necessary condition for the interpolating and 
Marcienkiewicz-Zygmund arrays in terms of the Beurling-Landau densities given 
by Theorem \ref{densitythmintro}.\\
\\
In what follows, when we write $A\lesssim B$, $A\gtrsim B$ or $A\simeq B$,
we mean that there are constants depending only on the manifold such that 
$A\leq C B$, $A\geq CB$ or $C_{1}B\leq A\leq C_{2}B$, respectively. Also, the 
value of the constants appearing during a proof may change but they will be still 
denoted with the same letter. A geodesic ball in $M$ and an Euclidean ball in $\R^m$ 
are represented by $B(\xi,r)$ and $\mathbb B(z,r)$, respectively.

\section{Kernels associated to $E_L$}\label{estimatesReprKernel}
Let
\[
K_{L}(z,w):=\sum^{k_{L}}_{i=1}\phi_{i}(z)\phi_{i}(w)=\sum_{\lambda_{i}\leq
L}\phi_{i}(z)\phi_{i}(w).
\]
This function is the reproducing kernel of the space $E_L$, i.e.
\begin{align*}
\forall f\in E_L\quad f(z)=\langle f,K_L(z,\cdot)\rangle.
\end{align*}
Note that $\dim(E_L)=k_{L}=\#\left\{\lambda_{i}\leq L\right\}.$ The function
$K_{L}$ is also called the spectral function associated to the Laplacian.
H\"ormander in \cite{Horm}, proved the following estimates.
\begin{enumerate}
  \item $K_{L}(z,z)=\frac{\sigma_{m}}{(2\pi)^{m}}L^{m}+O(L^{m-1})$
(uniformly in $z\in M$), where $\sigma_{m}=\frac{2\pi^{m/2}}{m\Gamma(m/2)}$.\\
  \item $k_{L}=\frac{\vol(M)\sigma_{m}}{(2\pi)^{m}}L^{m}+O(L^{m-1})$.
\end{enumerate}
In fact, in \cite{Horm}, there are estimates for the
spectral function associated to any elliptic operator of order $n\geq 1$ with
constants depending only on the manifold.\\
\\
So, for $L$ big enough we have $k_{L}\simeq L^{m}$ and
\[
\left\|K_{L}(z,\cdot)\right\|^{2}_{2}=K_{L}(z,z)\simeq L^{m}\simeq k_{L}
\]
with constants independent of $L$ and $z$.\\
\\
We will also use the Bochner-Riesz Kernel associated to the Laplacian that is defined as
\[
S^{N}_{L}(z,w):=\sum^{k_L}_{i=1}\left(1-\frac{\lambda_{i}}{L}\right)^{N}\phi_{i}(z)\phi_{i}(w).
\]
Here $N\in\mathbb N$ is the order of the kernel. Using the definition, one has that 
for all $g\in L^2(M)$, the Bochner-Riesz transform of $g$ is
\[
S^{N}_L(g)(z)=\int_{M}S^{N}_L(z,w)g(w)dV(w)=
\sum^{k_L}_{i=1}\left(1-\frac{\lambda_i}{L}\right)^Nc_i\phi_i(z)\in E_L,
\]
where $c_i=\langle g,\phi_i\rangle$. Observe that $\|S^N_L(g)\|_2\leq \|g\|_2$.\\
\\
Note that $S^{0}_{L}(z,w)=K_{L}(z,w)$. The Bochner-Riesz Kernel satisfies the following 
estimate.
\begin{equation}\label{RieszKernelEstimate}
|S^{N}_L(z,w)|\leq C_NL^{m}\left(1+Ld(z,w)\right)^{-N-1},
\end{equation}
where $C_N$ is a constant depending on the manifold and the order $N$.
This estimate has its origins in H\"ormander's article \cite[Theorem
5.3]{HormRiesz}. Estimate \eqref{RieszKernelEstimate} can be found also in
\cite[Lemma 2.1]{Sogge}.\\
\\
Note that on the diagonal, $S^{N}_{L}(z,z)\simeq C_{N}L^{m}$. The upper
bound is trivial by the definition and the lower bound follows from
\[
 S^N_{L}(z,z)\ge \sum_{\lambda_i\leq L/2}\left(1-\frac{\lambda_{i}}{L}\right)^{N}
\phi_{i}(z)\phi_{i}(z)\ge 2^{-N}K_{L/2}(z,z)\simeq C_{N}L^{m}.
\]
Similarly we observe that $\left\|S^{N}_{L}(\cdot,\xi)\right\|^{2}_{2}\simeq
C_{N}L^{m}$.\\
\\
We can consider other Bochner-Riesz type kernels. From now on, 
we fix an $\eps>0$ and $B^{\eps}_L$ will denote a transform from $L^{2}(M)$ 
to $E_L$ with kernel
\begin{equation}\label{defnewkernel}
B^{\eps}_L(z,w)=\sum^{k_L}_{i=1}\beta_{\eps}\left(\frac{\lambda_i}{L}\right)\phi_i(z)\phi_i(w),
\end{equation}
i.e.
\[
B^{\eps}_L(f)(z)=\int_{M}B^{\eps}_L(z,w)f(w)dV(w)
=\sum^{k_L}_{i=1}\beta_{\eps}\left(\frac{\lambda_i}{L}\right)\langle f,\phi_i\rangle \phi_i(z),
\]
where $\beta_{\eps}:[0,+\infty)\to [0,1]$ is a function of class $\mathcal C^{\infty}$ 
supported in $[0,1]$ such that $\beta_{\eps}(x)=1$ for $x\in [0,1-\eps]$ and $\beta_{\eps}(x)=0$ if 
$x\notin [0,1)$.
\[
\includegraphics[width=4.5cm]{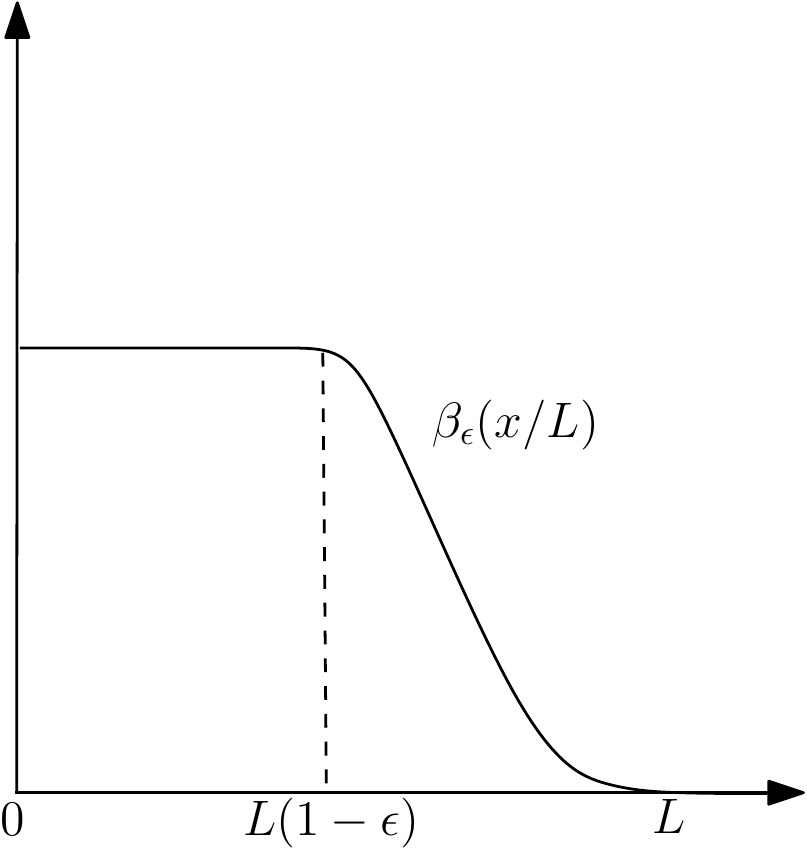}
\]
Observe that for $\eps=0$, the transform $B^0_L$ is just the orthogonal projection of 
the space $E_L$ i.e. the kernel $B^{0}_L(z,w)=K_L(z,w)$.\\
\\
We recall now an estimate for the kernel $B^{\eps}_L(z,w)$ that is similar to the Bochner-Riesz 
kernel estimate \eqref{RieszKernelEstimate}.
\begin{lemma}\label{kernelestimates}
Let $H:[0,+\infty)\to[0,1]$ be a function with continuous derivatives up to order $N\geq m$ 
with compact support in $[0,1]$. Then there exists a constant $C_N$ independent of $L$ 
such that
\begin{equation}\label{generalkernelestimate}
\left|\sum^{k_L}_{i=1}H(\lambda_i/L)\phi_{i}(z)\phi_i(w)\right|
\leq C_NL^{m}\frac{1}{(1+Ld(z,w))^{N}}, \quad \forall z,w\in M.
\end{equation}
\end{lemma}
For a proof see \cite[Theorem 2.1]{FM-Bernstein} and some ideas can be traced back from 
\cite{Sogge}.

\section{Definitions and Notations}
Given $L\geq 1$ and $m_{L}\in\mathbb N$, we consider a triangular family of 
points in $M$, $\Z=\left\{\Z(L)\right\}_L$, denoted as
\[
\mathcal Z(L)=\left\{z_{Lj}\in M\ :\ 1\leq j\leq m_{L}\right\},L\geq 1,
\]
and we assume that $m_{L}\to \infty$ as $L$ increases.

\begin{defini}
A triangular family of points $\mathcal Z$ in $M$ is \textbf{uniformly separated} 
if there exists a positive $\epsilon$ such that for all $L\geq 1$
$$d_{M}(z_{Lj},z_{Lk})\geq\frac{\epsilon}{L},\ \ j\neq k,$$
where $\eps$ is called the separation constant of $\Z$.
\end{defini}

\begin{remark}
The natural separation is of order $1/L$ in view of Proposition \ref{interpsep} 
(see below) that shows that a necessary condition for interpolation is that the 
family should be uniformly separated with this order of separation. The key 
idea is the Bernstein's inequality:
\begin{equation}\label{BernsteinIneqforEL}
\|\nabla f_L\|_{\infty}\lesssim L\|f_L\|_{\infty},\quad \forall f_L\in E_L.
\end{equation}
This estimate has been proved recently in \cite[Theorem 2.2]{FM-Bernstein}. 
Thus, on balls of radius $1/L$, a bounded function of $E_L$ oscillates little. 
\end{remark}

\begin{defini}\label{definitionMZ}
Let $\mathcal Z=\left\{\Z(L)\right\}_{L\geq 1}$ be a triangular family in $M$ with 
$m_{L}\geq k_{L}$ for all $L$. Then $\Z$ is a 
$L^{2}$-\textbf{Marcinkiewicz-Zygmund} (\textup{M-Z}) family, if there exists a constant 
$C>0$ such that for all $L\geq 1$ and $f_{L}\in E_{L}$
$$\frac{C^{-1}}{k_{L}}\sum^{m_{L}}_{j=1}|f_{L}(z_{Lj})|^{2}
\leq \int_{M}|f_{L}|^{2}dV\leq\frac{C}{k_{L}}\sum^{m_{L}}_{j=1}|f_{L}(z_{Lj})|^{2}.$$
\end{defini}
\begin{remark}
The condition of being M-Z can be expressed in terms of the reproducing kernel of 
$E_L$: a family $\Z$ is M-Z if and only if the normalized reproducing kernels form 
a frame with uniform bounds in $L$, i.e.
\[
\sum^{m_L}_{j=1}|\langle f_L,\tilde{K}_L(z_{Lj},\cdot)\rangle|^2\simeq \|f_L\|^2_2,
\]
with constants independent of $L$, where 
$\tilde{K}_L(z,w)=\frac{K_L(z,w)}{\|K_L(z,\cdot)\|_2}$.
\end{remark}
\begin{defini}\label{definitionInterp}
Let $\mathcal Z=\left\{\Z(L)\right\}_{L\geq 1}$ be a triangular family in $M$ with 
$m_{L}\leq k_{L}$ for all $L$. Then $\Z$ is a $L^{2}$-\textbf{interpolating family} 
if for all family of values $c=\left\{c(L)\right\}_{L\geq 1}$,  
$c(L)=\left\{c_{Lj}\right\}_{1\leq j\leq m_{L}}$ such that
$$\sup_{L\geq 1}\frac{1}{k_{L}}\sum^{m_{L}}_{j=1}|c_{Lj}|^{2}<\infty,$$
there exists a sequence of functions $f_{L}\in E_{L}$ with
\[
\sup_{L\geq 1}\left\|f_L\right\|_{2}<\infty
\]
and $f_{L}(z_{Lj})=c_{Lj}$ ($1\leq j\leq m_{L}$). That is, 
$f_{L}(\Z(L))=c(L)$ for all $L\geq 1$.
\end{defini}
\begin{remark}
Equivalently, a family is interpolating if the normalized reproducing kernels form a 
Riesz sequence, i.e.
\[
\frac{1}{k_L}\sum_{j}|c_{Lj}|^2\simeq \|\sum_{j}c_{Lj}\tilde{K}_{L}(z_{Lj},\cdot)\|^2_2,
\]
with constants independent of $L$, whenever $c=\left\{c_{Lj}\right\}_{j,L}$ is a family 
satisfying
\[
\sup_{L}\frac{1}{k_L}\sum^{m_L}_{j=1}|c_{Lj}|^2<\infty.
\]
\end{remark}
Intuitively, a M-Z family should be \textit{dense} in order to recover the $L^2$-norm of functions 
of the space $E_L$ and an interpolating family should be \textit{sparse}.

We recall a result that is a Plancherel-P\'olya type inequality. For a proof, 
see \cite[Theorem 4.6]{OP}.

\begin{thm}[Plancherel-P\'olya Theorem]\label{unionsepfam}
Let $\Z$ be a triangular family of points in $M$, i.e. 
$\Z=\left\{z_{Lj}\right\}_{j\in\left\{1,\ldots,m_{L}\right\},L\geq 1}\subset
M$. Then $\Z$ is a finite union of
uniformly separated families, if and only if there exists a constant $C>0$ such
that for all $L\geq 1$ and $f_{L}\in E_{L}$,
\begin{equation}\label{PlancherelPolya}
\frac{1}{k_{L}}\sum^{m_{L}}_{j=1}|f_{L}(z_{Lj})|^{2}\leq
C\int_{M}|f_{L}(\xi)|^{2}dV(\xi).
\end{equation}
\end{thm}
\begin{remark}
The above result is interesting because the inequality \eqref{PlancherelPolya}
means that the sequence of normalized reproducing kernels 
is a Bessel sequence for $E_{L}$, i.e.
\begin{align*}
\sum^{m_L}_{j=1}|\langle f,\tilde{K}_{L}(\cdot,z_{Lj})\rangle|^2\lesssim
\left\|f\right\|^2_{2}\quad \forall f\in E_{L},
\end{align*}
where $\left\{\tilde{K}_{L}(\cdot,z_{Lj})\right\}_{j}$ are the 
normalized reproducing kernels. 
Note that $|\tilde{K}_{L}(\cdot,z_{Lj})|^2\simeq
|K_{L}(\cdot,z_{Lj})|^2 k_{L}^{-1}$. That's the reason why the
quantity $k_{L}$ appears in the inequality \eqref{PlancherelPolya} and in 
Definitions \ref{definitionMZ} and \ref{definitionInterp}. 
\end{remark}

\section{Interpolating and M-Z families}
In this section we will present some qualitative results about the interpolating and M-Z 
families.

\subsection{Interpolating families}
The following result shows that the interpolation can be done in a stable way.
\begin{lemma}\label{propinterpfunc}
Let $\Z$ be a triangular family in $M$. Assume $\Z$ is interpolating. 
Then the interpolation can be done by functions $f_{L}\in E_{L}$ such that
\[
\left\|f_{L}\right\|^{2}_{2}\leq \frac{C}{k_{L}}\sum^{m_{L}}_{j=1}|f_{L}(z_{Lj})|^{2},
\]
where $C$ is independent of $L$.
\end{lemma}
The proof follows from the Closed Graph's Theorem (check the basic ideas in 
\cite[Proposition 2, Page 129]{Young}).\\
\\
Now, we provide a necessary condition for an interpolating family.
\begin{prop}\label{interpsep}
Let $\Z$ be an $L^2$-interpolating triangular family in $M$. 
Then $\Z$ is uniformly separated.
\end{prop}
\begin{proof}
Fix $L_{0}\geq 1$ and $1\leq j_{0}\leq m_{L_{0}}$. Using Lemma \ref{propinterpfunc}, 
there exist functions $f_{L_{0}}\in E_{L_{0}}$ such that 
$f_{L_{0}}(z_{L_{0}j})=\delta_{jj_{0}}$ and $\left\|f_{L_{0}}
\right\|^{2}_{2}\leq C/k_{L_{0}}$ ($C$ independent of $L$). Applying 
\cite[Proposition 3.4]{OP}, we get the following.
\begin{align*}
1& =|f_{L_{0}}(z_{L_{0}j_{0}})-f_{L_{0}}(z_{L_{0}j})|\leq 
\left\|\nabla f_{L_{0}}\right\|_{\infty}d_{M}(z_{L_{0}j_{0}},z_{L_{0}j})\\
& \lesssim \sqrt{k_{L_{0}}}L_{0}\left\|f_{L_{0}}\right\|_{2}
d_{M}(z_{L_{0}j_{0}},z_{L_{0}j})\lesssim 
L_{0}\sqrt{k_{L_{0}}}\frac{1}{\sqrt{k_{L_{0}}}}d_{M}(z_{L_{0}j_{0}},z_{L_{0}j})\\
& \simeq L_{0}d_{M}(z_{L_{0}j_{0}},z_{L_{0}j}).
\end{align*}
Thus,
\[
d_{M}(z_{L_{0}j_{0}},z_{L_{0}j})\gtrsim 
\frac{1}{L_{0}},\quad \forall L_{0}\geq 1,\ j\neq j_{0},
\]
where the constant does not depend on $L_{0}$ and $j_0$.
\end{proof}
\begin{thm}
Let $\Z$ and $\Z'$ be two triangular families in $M$. Assume that $\Z$ is an 
$L^2$-interpolating family. Then there exists $\delta_0>0$ such that for 
all $0<\delta\leq \delta_0$, $Z'$ is also $L^2$-interpolating provided
\[
d_M(z'_{Lj},z_{Lj})<\delta/L, \quad \forall j=1,\cdots,m_L; L\geq 1.
\]
\end{thm}
The proof follows by a perturbation argument and some gradient estimates proved 
in \cite{OP}. The next proposition gives us a sufficient condition for interpolation. It 
says, essentially, that a \textit{sparse} family should be interpolating.

\begin{prop}
Let $\Z=\left\{\Z(L)\right\}_{L}=\left\{z_{Lj}\right\}_{L\geq 1,j=1,...,m_{L}}\subset M$ 
be a triangular family of points with $m_L\leq k_{L}$. Assume $\Z$ is separated 
enough, i.e. there exists $R>0$ (big enough) such that
\[
d(z_{Lj},z_{Lk})\geq \frac{R}{L},\quad \forall j\neq k, \quad \forall L.
\]
Then $\Z$ is an interpolating family.
\end{prop}

\begin{proof}
In the following, we consider the Banach spaces:
\[
\A=\left\{v=\left\{v_L\right\}_{L}=\left\{v_{Lj}\right\}^{m_L}_{j=1}:
\quad \|v\|^2_{\A}:=\sup_{L}\frac{1}{k_L}\sum^{m_L}_{j=1}|v_{Lj}|^2<\infty\right\}
\]
and $E:=\left\{f=(f_L)_L;\quad f_L\in E_L\right\}$ endowed with the norm
\[
\|f\|^2_{E}=\sup_{L}\|f_L\|^2_{2}.
\]

Let $\mathcal R:\ E\to \A$ be the evaluating operator, i.e. if 
$v:=\mathcal R(f)$ for some $f\in E$, 
then $v_{Lj}=f_{L}(z_{Lj})$. This operator is linear and continuous by the 
Plancherel-P\'olya type inequality (Theorem \ref{unionsepfam}). 
Now, consider the operator $\mathcal S:\A\to E$ defined as follows: if $v\in\A$, then 
$\mathcal S(v)=:f$ with
\[
f_L(z):=\sum^{m_L}_{j=1}v_{Lj}\frac{S^{N}_{L}(z_{Lj},z)}{S^{N}_{L}(z_{Lj},z_{Lj})},
\]
where $S^{N}_{L}(z,w)$ is the Bochner-Riesz Kernel of order $N$ associated to the 
Laplacian (see Section \ref{estimatesReprKernel} for the definition). The order 
$N$ will be chosen later. Note that the functions $f_L$ belong to $E_L$ and
\[
f_L(z_{Lk})=v_{Lk}+\sum_{j\neq k}v_{Lj}
\frac{S^{N}_{L}(z_{Lj},z_{Lk})}{S^{N}_{L}(z_{Lj},z_{Lj})}.
\]
The operator $\mathcal S$ is well defined. Indeed, let $v\in\A$ and $f:=\mathcal S(v)$. 
We need to prove that $f\in E$. Using Cauchy-Schwarz inequality, we obtain:
\begin{align*}
\left\|f_L\right\|_{2} & =\sup_{\left\|g\right\|_{2}=1}|\langle f_L,g\rangle|=
\sup_{\left\|g\right\|_{2}=1}\left|\sum^{m_{L}}_{j=1}v_{Lj}
\frac{\langle S^{N}_{L}(z_{Lj},\cdot),g\rangle}{S^{N}_{L}(z_{Lj},z_{Lj})}\right|\\
& \lesssim \|v\|_{\A}\sup_{\|g\|_{2}=1}\|S^{N}_{L}g\|_2\leq \|v\|_{\A},
\end{align*}
where we have applied Theorem \ref{unionsepfam} to $S^{N}_L(g)$. Therefore, 
$\|f\|_{E}\lesssim \|v\|_{\A}<\infty$. This proves that $\mathcal S$ is well defined and 
continuous. Obviously this operator is linear.\\
\\
If $\left\|\mathcal R\circ \mathcal S-Id\right\|<1$, then $\mathcal R$ is invertible. Furthermore, $\mathcal R$ is exhaustive 
and as a consequence the family $\Z$ is interpolating. We only need to check that 
$\left\|\mathcal R\circ \mathcal S-Id\right\|<1$. We claim that
\begin{equation}\label{claimenoughseparated}
\sum_{j\neq k}\left|\frac{S^{N}_{L}(z_{Lj},z_{Lk})}{S^{N}_{L}(z_{Lj},z_{Lj})}\right|<<1,
\end{equation}
uniformly in $L$ for $R$ big enough, provided $N+1>m$. Thus,
\[
\|\mathcal R\circ \mathcal S-Id\|^2=\sup_{v\in \A;\|v\|_{\A}=1}\|\mathcal R(\mathcal S(v))-v\|^2_{\A}=
\sup_{v\in\A; \|v\|_{\A}=1}\|w\|^2_{\A},
\]
where $w=\left\{w_{Lk}\right\}_{k;L}$ with
\[
w_{Lk}=\sum_{j\neq k}v_{Lj}\frac{S^{N}_{L}(z_{Lj},z_{Lk})}{S^{N}_{L}(z_{Lj},z_{Lj})}.
\]
Using the claim \eqref{claimenoughseparated}, we get a control of the $L^\infty$-norm 
of $w$:
\[
\sup_{L}|w_{Lk}|\leq \sup_{L}
\sum_{j\neq k}\left|\frac{S^{N}_{L}(z_{Lj},z_{Lk})}{S^{N}_{L}(z_{Lj},z_{Lj})}\right|<<1,
\]
for all $v=\left\{v_{Lj}\right\}_{j;L}$ such that $\sup_{L}|v_{Lj}|=1$. Moreover, using 
again \eqref{claimenoughseparated}, we have the same control of the $L^1$-norm of $w$:
\begin{align*}
\sup_{L}\frac{1}{k_L}\sum^{m_L}_{k=1}|w_{Lk}|
\simeq\sup_{L}\frac{1}{k_L}\sum^{m_L}_{j=1}|v_{Lj}|\sum_{k\neq j}
\left|\frac{S^{N}_{L}(z_{Lk},z_{Lj})}{S^{N}_{L}(z_{Lk},z_{Lk})}\right|<<1,
\end{align*}
for all $v=\left\{v_{Lj}\right\}_{j;L}$ such that $\sup_{L}\frac{1}{k_L}\sum_{j}|v_{Lj}|=1$.
Thus, interpolating between the $L^1$-norm and $L^{\infty}$-norm, we get the same 
result for the $L^2$-norm of $w$ and the proof is complete. Now we proceed in order 
to prove the claim \eqref{claimenoughseparated}. Let
\[
g_{k}(z):=\frac{1}{(1+Ld_M(z,z_{Lk}))^{N+1}}
\]
and $B_{j}:=B(z_{Lj},1/L)$. It is easy to check that
\[
\inf_{B_{j}}g_{k}(z)\geq \frac{1}{2^{N+1}}g_{k}(z_{Lj}).
\]
Using the fact that $\Z$ is separated enough, we know that $B_{j}$ 
are pairwise disjoint and 
$\cup_{j\neq k}B_{j}\subset M\setminus B(z_{Lk},(R-1)/L)$. Therefore, applying 
\eqref{RieszKernelEstimate},
\begin{align*}
& \sum_{j\neq k}\left|\frac{S^{N}_{L}(z_{Lj},z_{Lk})}{S^{N}_{L}(z_{Lj},z_{Lj})}\right|
\leq C_{N}\sum_{j\neq k}g_{k}(z_{Lj})
\leq C_NL^{m}\sum_{j\neq k}\int_{B_{j}}g_{k}(z)dV(z)\\
& \leq C_{N}L^{m}\int_{M}\frac{1}{(Ld_M(z,z_{Lk}))^{N+1}}
\chi_{M\setminus B(z_{Lk},(R-1)/L)}(z)dV.
\end{align*}
Consider for any $t\geq 0$, the following set $A_{t}$.
\[
A_{t}=\left\{z\in M:\quad d_M(z,z_{Lk})\geq \frac{R-1}{L},
\quad d_M(z,z_{Lk})<\frac{t^{-1/(N+1)}}{L}\right\}.
\]
Using the distribution function, one can compute that
\begin{align*}
\sum_{j\neq k}\left|\frac{S^{N}_{L}(z_{Lj},z_{Lk})}{S^{N}_{L}(z_{Lj},z_{Lj})}\right|
& \leq C_{N}L^{m}\int^{(R-1)^{-(N+1)}}_{0}\vol(A_{t})dt\\
& \lesssim C_{N}\frac{1}{(R-1)^{(N+1)-m}},
\end{align*}
provided $N+1>m$. Taking $R$ big enough 
we get the desired claim.
\end{proof}

\subsection{Marcinkiewicz-Zygmund families}
In what follows, we present some qualitative results concerning the M-Z 
families. The proof of these results follows from standard techniques and the 
ideas in \cite[Theorem 4.7]{Mar}, replacing the corresponding gradient 
estimates obtained in \cite{OP}.\\
\\
The following theorem allows us to assume, without loss of generality, 
that a M-Z familiy is uniformly separated.
\begin{thm}\label{MZseparated}
Let $\Z\subset M$ be an $L^2$-\textup{M-Z} family. Then there exists a uniformly 
separated family $\tilde{\Z}\subset \Z$ which is also an $L^2$-\textup{M-Z} family.
\end{thm}
The next result shows us that a small perturbation of a M-Z family is still a M-Z 
family.
\begin{thm}
Let $\mathcal Z$ be an $L^2$-\textup{M-Z} family. There exists $\eps_0>0$ such 
that if $\mathcal Z'$ is a uniformly separated family with
\[
d_{M}(z_{Lj},z'_{Lj})<\frac{\eps}{L},
\]
for some $\eps\leq \eps_0$, then the family of points $\Z'$ is $L^2$-\textup{M-Z}.
\end{thm}
Now we provide a sufficient condition for a family to be $L^2$-M-Z. 
Intuitively, a family should be \textit{dense} in order to be M-Z.
\begin{thm}\label{sufficientconditionMZ}
There exists $\eps_0>0$ such that if $\mathcal Z$ is an $\eps$-dense family 
(not necessarily uniformly separated), i.e. for all $L\geq1$
\[
\sup_{\xi\in M}d_{M}(\xi,\mathcal Z(L))<\frac{\eps}{L}, \quad (\eps\leq \eps_0),
\]
then there exists a uniformly separated subfamily which is $\tilde{\eps}$-dense and 
is an $L^2$-\textup{M-Z} family provided that $\tilde{\eps}\leq \eps_0$.
\end{thm}
\begin{remark}
Theorem \ref{sufficientconditionMZ} has been also proved by F. Fibir and 
H.N. Mhaskar using other techniques (see \cite[Theorem 5.1]{FM-MZ}).
\end{remark}

\section{Beurling-Landau density}\label{densityresultsgeneralmfolds}

In this section, we provide necessary conditions for a family to be interpolating 
or M-Z in terms of the following Beurling-Landau type densities.

\begin{defini}\label{defBeurlingLandauDensity}
Let $\Z$ be a triangular family of points in $M$. We define the upper and lower 
Beurling-Landau density, respectively, as
\[
D^{+}(\Z)=\limsup_{R\to\infty}\left(\limsup_{L\to\infty}
\max_{\xi\in M}\frac{\frac{1}{k_L}\#(\Z(L)\cap B(\xi,R/L))}
{\frac{\vol(B(\xi,R/L))}{\vol(M)}}\right),
\]
\[
D^{-}(\Z)=\liminf_{R\to\infty}\left(\liminf_{L\to\infty}
\min_{\xi\in M}\frac{\frac{1}{k_L}\#(\Z(L)\cap B(\xi,R/L))}
{\frac{\vol(B(\xi,R/L))}{\vol(M)}}\right).
\]
\end{defini}
\begin{remark}
Let $\mu_L$ be the normalized counting measure, i.e.
\[
\mu_L=\frac{1}{k_L}\sum^{m_L}_{j=1}\delta_{z_{Lj}}
\]
and $\sigma$ the normalized volume measure, i.e. $d\sigma=dV/\vol(M)$. 
Then the densities defined above can be viewed as the asymptotic behaviour 
of the quantity
\[
\frac{\mu_L(B(\xi,R/L))}{\sigma(B(\xi,R/L))}.
\]
\end{remark}
Our main result is:
\begin{thm}\label{charactdensitygeneral}
Let $M$ be an arbitrary smooth compact Riemannian manifold without boundary 
of dimension $m\geq 2$ and $\Z$ a triangular family in $M$. If $\Z$ is an 
$L^2$-\textup{M-Z} family then there exists a uniformly separated $L^2$-\textup{M-Z} family 
$\tilde{\Z}\subset \Z$ such that
\[
D^{-}(\tilde{\Z})\geq 1.
\]
If $\Z$ is an $L^2$-interpolating family then it is uniformly separated and
\[
D^{+}(\Z)\leq 1.
\]
\end{thm}
This result was proved in the particular case when $M=\mathbb S^{m}$ in \cite{Mar}. 
Following the ideas of \cite{Mar}, we prove Theorem \ref{charactdensitygeneral} 
in the general case of a manifold. In \cite{Mar}, the key idea to prove this result was 
the comparision of the trace of the concentration operator and its square with an 
estimate of the eigenvalues of this operator. In general, the main difference from the 
case of the Sphere is that we lack of an explicit expression of the reproducing kernel. 
Thus, in the general setting, we need to work with a \lq\lq modified" concentration 
operator. Before we proceed, we shall introduce the concept of the classical and 
modified concentration operator.
\subsection{Classical Concentration Operator}
\begin{defini}
The \textit{classical concentration operator} $\K^{L}_A$, over a set $A\subset M$, is defined for 
$f_L\in E_L$ as
\begin{equation}
\K^{L}_Af_L(z)=\int_{A}K_L(z,\xi)f_L(\xi)dV(\xi).
\end{equation}
\end{defini}
This operator is the composition of the restriction operator to $A$ with the 
orthogonal projection of $E_L$, i.e. $\K^{L}_{A}(f_L)=P_{E_L}(\chi_{A}f_L)$ for all 
$f_L\in E_L$. The operator $\K^{L}_A$ is self-adjoint. Indeed, if $f_L,g_L\in E_L$ 
then:
\begin{align*}
\langle \K^{L}_Af_L,g_L \rangle&=\langle P_{E_L}(\chi_A\cdot f_L),g_L \rangle=
\langle \chi_A\cdot f_L,P_{E_L}(g_L)\rangle=\langle \chi_A\cdot f_L,g_L\rangle\\
&=\langle f_L,\chi_A\cdot g_L\rangle=\langle P_{E_L}f_L,\chi_A\cdot g_L\rangle
=\langle f_L,\K^L_A(g)\rangle.
\end{align*}
Alternatively, we can view the action of the concentration operator as a matrix acting on 
a sequence $\beta=\left\{\beta_{i}\right\}_{i=1,\ldots,k_L}$ that are the Fourier coefficients of a 
function $f_L\in E_L$ (with respect to the orthonormal basis $\left\{\phi_i\right\}$). If 
we denote by $D_L := (d_{ij})^{k_L}_{i,j=1}$, where 
\[
d_{ij}=\int_{A}\phi_i\phi_j,
\]
then $\K^{L}_A(f_L)\equiv D_L(\beta)$.\\
Using the spectral theorem, we know that the eigenvalues of $\K^L_A$ are all real and 
$E_L$ has an orthonormal basis of eigenvectors of $\K^{L}_A$. The trace of $\K^{L}_A$ is
\[
\text{tr}(\K^{L}_{A})=\sum^{k_L}_{i=1}d_{ii}=\int_{A}K_{L}(z,z)dV(z).
\]
Similarly, we can compute the trace of $\K^{L}_A\circ \K^{L}_{A}$.
\begin{align*}
\text{tr}(\K^{L}_A\circ \K^{L}_A)=\sum^{k_L}_{i,j=1}d_{ij}d_{ji}
=\int_{A\times A}|K_{L}(z,w)|^2dV(w)dV(z).
\end{align*}
We will choose $A$ as $B(\xi, R/L)$ for some fixed point $\xi\in M$. Taking into 
account that
\[
\vol(B(\xi,R/L))\simeq \frac{R^m}{L^{m}}
\]
and using H\"ormander's estimates for the reproducing kernel and $k_L$ (see 
Section \ref{estimatesReprKernel}), we get
\begin{equation}\label{traceestimatetoapply}
\text{tr}(\K^{L}_{B(\xi,R/L)})=k_L\frac{\vol(B(\xi,R/L))}{\vol(M)}+
\frac{o(L^{m})}{L^{m}}.
\end{equation}

\subsection{Modified Concentration Operator}
From now on, we fix an $\eps>0$ and consider the transform $B^{\eps}_L$ defined 
in Section \ref{estimatesReprKernel} associated with the kernel
\[
B^{\eps}_L(z,w)=\sum^{k_L}_{i=1}\beta_{\eps}\left(\frac{\lambda_i}{L}\right)\phi_i(z)\phi_i(w),
\]
i.e. for all $f\in L^2(M)$,
\[
B^{\eps}_L(f)(z)=\int_{M}B^{\eps}_L(z,w)f(w)dV(w)
=\sum^{k_L}_{i=1}\beta_{\eps}\left(\frac{\lambda_i}{L}\right)\langle f,\phi_i\rangle\phi_i(z).
\]
\begin{defini}
The \textit{modified concentration operator} $T^{\eps}_{L,A}$, over a set $A\subset M$, is defined 
for $f_L\in E_L$ as:
\[
T^{\eps}_{L,A}f_L(z)=B^{\eps}_L(\chi_{A}\cdotÊB^{\eps}_L(f_L))(z)
=\int_{M}B^{\eps}_L(z,w)\chi_{A}(w)B^{\eps}_L(f_L)(w)dV(w).
\]
\end{defini}
Observe that for $\eps=0$, the modified concentration operator is just the classical 
concentration operator defined previously.\\
\\
An advantage of $T^{\eps}_{L,A}$ in contrast of $\K^{L}_{A}$ is that we have a nice estimate 
of its kernel: using Lemma \ref{kernelestimates}, we know that for any $N\geq m$, there exists 
a constant $C_N$ independent of $L$ such that
\[
|B^{\eps}_{L}(z,w)|\leq C_NL^{m}\frac{1}{(1+Ld(z,w))^{N}}, \quad \forall z,w\in M.
\]
The operator $T^{\eps}_{L,A}$ is self-adjoint and by the spectral theorem 
its eigenvalues are all real and $E_L$ has an orthonormal basis of eigenvectors of 
$T^{\eps}_{L,A}$. In fact, the main reason to do the first smooth projection in 
$T^{\eps}_{L,A}$ is to ensure the self-adjointness of the operator (but the calculations work 
even if we consider only $B^{\eps}_L(\chi_A\cdot)$).\\
\\
As before, we can compute the trace of $T^{\eps}_{L,A}$ and 
$T^{\eps}_{L,A}\circ T^{\eps}_{L,A}$ that will be used later on.
\begin{align*}
\tr(T^{\eps}_{L,A})=\sum^{k_L}_{i=1}\beta_{\eps}^2\left(\frac{\lambda_i}{L}\right)\int_{A}\phi^2_i(z)dV(z)
=:\int_{A}\tilde{B}^{\eps}_L(z,z)dV(z),
\end{align*}
where $\tilde{B}^{\eps}_L(z,w)$ is a kernel defined as
\[
\tilde{B}^{\eps}_L(z,w)=\sum^{k_L}_{i=1}\alpha\left(\frac{\lambda_i}{L}\right)\phi_i(z)\phi_i(w),
\]
with $\alpha(x):=\beta_{\eps}^2(x)$. Note that the function $\alpha$ has the same properties as $\beta_{\eps}$ 
and therefore we know that $\tilde{B}^{\eps}_L(z,w)$ has the estimate \eqref{generalkernelestimate}.\\
\\
Similarly we can compute the trace of $T^{\eps}_{L,A}\circ T^{\eps}_{L,A}$.
\begin{align*}
\tr(T^{\eps}_{L,A}\circ T^{\eps}_{L,A})
=\int_{A\times A}|\tilde{B}^{\eps}_L(z,w)|^2dV(z)dV(w).
\end{align*}
Since the modified concentration operator is a \textit{small} perturbation of $\K^{L}_{A}$, one 
can estimate $\tr(T^{\eps}_{L,A})$ in terms of $\tr(\K^{L}_{A})$. Indeed, using the definition 
of $\beta_{\eps}(x)$, 
\[
\tr(\mathcal K^{L(1-\eps)}_{A})\leq \tr(T^{\eps}_{L,A})\leq \tr(\mathcal K^{L}_{A}).
\]
Applying this computation to $A=A_{L}:=B(\xi,R/L)$ and using \eqref{traceestimatetoapply}, 
we get the following.
\[
\frac{\tr(T^{\eps}_{L,B(\xi,R/L)})}{k_L\frac{\vol(B(\xi,R/L))}{\vol(M)}}\geq 
\frac{k_{L(1-\eps)}\frac{\vol(B(\xi,R/L))}{\vol(M)}}{k_L\frac{\vol(B(\xi,R/L))}{\vol(M)}}
+\frac{o(L^{m}(1-\eps)^{m})}{L^{m}(1-\eps)^{m}}\frac{1}{k_L\frac{\vol(B(\xi,R/L))}{\vol(M)}}.
\]
Since $\vol(B(\xi,R/L))\simeq R^{m}/L^{m}$, the second term tends to 0 when $L\to \infty$. 
Thus, using the expression for $k_L$ (see Section \ref{estimatesReprKernel}), we get:
\begin{equation}\label{tracequotientbelow}
\liminf_{L\to\infty}\frac{\tr(T^{\eps}_{L,A_L})}{k_L\frac{\vol(B(\xi,R/L))}{\vol(M)}}
\geq (1-\eps)^{m}, \quad \forall \eps>0.
\end{equation}
The upper bound for this quantity is trivial since 
$\tr(T^{\eps}_{L,A_L})\leq \tr(\mathcal K^{L}_{A_L})$ and 
has been computed previously. Hence, using \eqref{traceestimatetoapply} we have
\begin{equation}\label{tracequotientabove}
\limsup_{L\to\infty}\frac{\tr(T^{\eps}_{L,A_L})}{k_L\frac{\vol(B(\xi,R/L))}{\vol(M)}}\leq 1.
\end{equation}
Similarly, if $\rho>0$ is a fixed number, then
\begin{equation}\label{modifiedtracequotientabove}
\limsup_{L\to\infty}\frac{\tr(T^{\eps}_{L(1+\rho),A_L})}{k_L\frac{\vol(B(\xi,R/L))}{\vol(M)}}
\leq (1+\rho)^{m}.
\end{equation}

\subsection{Proof of the main result}
In the spirit of the original work of Landau, the proof of Theorem \ref{charactdensitygeneral} 
relies on a trace estimate of $T^{\eps}_{L,A}$ and two technical lemmas 
(Lemma \ref{controlvapsconcentraciosamplinggeneral} and 
\ref{controlvapsconcentraciointerpolaciogeneral} below) that estimate the number of big 
eigenvalues of the modified concentration operator. First we state these 
technical results and show the proof of the main result and in Sections \ref{tracegeneral} 
and \ref{technicalresultsgeneral} we present a proof of them.\\
\\
The following result is an estimate of the difference of the trace of $T^{\eps}_{L,A_L}$ 
and $T^{\eps}_{L,A_L}\circ T^{\eps}_{L,A_L}$. It will show us, later on, that most of the 
eigenvalues are either close to 1 or to 0.
\begin{prop}\label{traceestimategeneral}
Let $A_L=B(\xi,R/L)$. Then
\[
\limsup_{L\to\infty}\left(\tr(T^{\eps}_{L,A_L})-\tr(T^{\eps}_{L,A_L}
\circ T^{\eps}_{L,A_L})\right)\leq C_{1}(1-(1-\eps)^{m})R^{m}+C_{2}R^{m-1},
\]
where $C_{1}$ (independent of $\eps$) and $C_{2}$ are constants independent of 
$L$ and $R$.\\
Similarly, if $\rho>0$ then
\begin{align*}
\limsup_{L\to\infty} & 
\left(\tr(T^{\eps}_{L(1+\rho),A_L})-\tr(T^{\eps}_{L(1+\rho),A_L}
\circ T^{\eps}_{L(1+\rho),A_L})\right)\\
& \leq C_{1}(1+\rho)^{m}(1-(1-\eps)^{m})R^{m}+C_{2}R^{m-1},
\end{align*}
where $C_{1}$ (independent of $\eps$ and $\rho$) and $C_{2}$ are constants 
independent of $L$ and $R$.
\end{prop}

Given $L\geq 1$ and $R>0$, let $A_L,A^{+}_L=A^{+}_L(t)$ and $A^{-}_L=A^{-}_L(t)$ be the balls centered 
at a fixed point $\xi\in M$ and radius $R/L$, $(R+t)/L$ and $(R-t)/L$, respectively, 
where $t$ is a parameter such that $s<<t<<R<<L$ and $s$ is the separation constant of the family $\Z$. The value of 
$t$ will be chosen later on. We denote the eigenvalues of the 
modified concentration operator $T^{\eps}_{L,A_L}$ as
\[
1>\lambda^{L}_{1}\geq \cdotsÊ\geq \lambda^{L}_{k_L}>0.
\]
\begin{lemma}\label{controlvapsconcentraciosamplinggeneral}
Let $\Z$ be an $s$-uniformly separated $L^{2}$-\textup{M-Z} family. Then there exist $t_{0}=t_{0}(M,s)>0$ 
and a constant $0<\gamma<1$ (independent of $\eps$, $R$ and $L$) 
such that for all $t\geq t_{0}$,
\[
\lambda^{L}_{N_L+1}\leq \gamma,
\]
where
\[
N_L:=N_L(t)=\#(\Z(L)\cap A^{+}_L)=\#(\Z(L)\cap B(\xi, (R+t)/L)).
\]
\end{lemma}
\begin{remark}\label{remarksamplinggeneral}
In the conditions of Lemma \ref{controlvapsconcentraciosamplinggeneral},
\[
\#\left\{\lambda^{L}_{j}>\gamma\right\}\leq N_L=\#(\Z(L)\cap A^{+}_L)\leq 
\#(\Z(L)\cap A_L)+O(R^{m-1}), R\to\infty,
\]
where the constant in $O(R^{m-1})$ does not depend on $L$.
\end{remark}
\begin{proof}
The first inequality is trivial by Lemma \ref{controlvapsconcentraciosamplinggeneral} and 
the second inequality follows using the separation of the family $\Z$. Moreover, 
$N_L\lesssim R^m/s^m$.
\end{proof}

\begin{lemma}\label{controlvapsconcentraciointerpolaciogeneral}
Let $\Z$ be an $L^{2}$-interpolating family with separation constant $s$ and $\rho>0$. Then there exist
$t_{1}=t_{1}(M,s)>0$ and a constant $0<\delta<1$ independent of $\eps$, $R$ and $L$ 
such that for all $t\geq t_{1}$,
\[
\lambda^{L(1+\rho)}_{n_L-1}\geq \delta:=C\beta_{\eps}^2\left(\frac{1}{1+\rho}\right),
\]
where $\lambda^{L(1+\rho)}_{k}$ are the eigenvalues associated to $T^{\eps}_{L(1+\rho),A_L}$, 
$C$ is independent of $\rho$ and $\eps$ and 
\[
n_L:=n_L(t)=\#(\Z(L)\cap A^{-}_{L})=\#(\Z(L)\cap B(\xi,(R-t)/L)).
\]
\end{lemma}

\begin{remark}\label{remarkinterpolaciogeneral}
In the conditions of Lemma \ref{controlvapsconcentraciointerpolaciogeneral} we have
\[
\#(\Z(L)\cap A_L)-O(R^{m-1})\leq n_L=\#(\Z(L)\cap A^{-}_L)
\leq \#\left\{\lambda^{L(1+\rho)}_{j}\geq \delta\right\}+1,
\]
where the constant in $O(R^{m-1})$ does not depend on $L$.
\end{remark}

\begin{proof}
The second inequality is trivial by Lemma \ref{controlvapsconcentraciointerpolaciogeneral} 
and the first inequality follows using the separation of $\Z$.
\end{proof}

In what follows, we pick the parameter $t$ in the range $\max(t_{0},t_{1})\leq t<<R$, where $t_{0}$ 
and $t_{1}$ are the values given by Lemmas \ref{controlvapsconcentraciosamplinggeneral} and 
\ref{controlvapsconcentraciointerpolaciogeneral}.\\
\\
Now we have all the tools in order to prove the main result concerning the notion 
of densities.

\begin{proof}[Proof of Theorem \ref{charactdensitygeneral}]
Assume $\Z$ is an $L^2$-\textup{M-Z} family. Without loss of generality, we may assume 
that $\Z$ is uniformly separated (see Theorem \ref{MZseparated}). 
Consider the following measures:
\[
d\mu_L=\sum^{k_L}_{j=1}\delta_{\lambda^{L}_{j}}.
\]
Note that
\[
\text{tr}(T^{\eps}_{L,A_L})=\int^{1}_{0}xd\mu_L(x),\quad
\text{tr}(T^{\eps}_{L,A_L}\circ T^{\eps}_{L,A_L})=\int^{1}_{0}x^2d\mu_L(x).
\]
Let $\gamma$ be given by Lemma \ref{controlvapsconcentraciosamplinggeneral}. We have
\begin{align*}
\#\left\{\lambda^{L}_{j} >\gamma\right\} & =\int^{1}_{\gamma}d\mu_L(x)\geq
\int^{1}_{0}xd\mu_L(x)-\frac{1}{1-\gamma}\int^{1}_{0}x(1-x)d\mu_L(x)\\
& = \text{tr}(T^{\eps}_{L,A_L})
-\frac{1}{1-\gamma}(\text{tr}(T^{\eps}_{L,A_L})-\text{tr}(T^{\eps}_{L,A_L}\circ T^{\eps}_{L,A_L})),
\end{align*}
Using the remark following Lemma \ref{controlvapsconcentraciosamplinggeneral} and 
\eqref{tracequotientbelow}, we have
\begin{align*}
& \liminf_{L\to\infty}\frac{\#(\Z(L)\cap A_L)
+O(R^{m-1})}{k_L\frac{\vol(B(\xi,R/L))}{\vol(M)}}\\
& \geq \liminf_{L\to\infty}\left[\frac{\text{tr}(T^{\eps}_{L,A_L})}
{k_L\frac{\vol(B(\xi,R/L))}{\vol(M)}}
-\frac{1}{1-\gamma}\frac{\text{tr}(T^{\eps}_{L,A_L})-\text{tr}(T^{\eps}_{L,A_L}\circ T^{\eps}_{L,A_L})}
{k_L\frac{\vol(B(\xi,R/L))}{\vol(M)}}\right]\\
& \geq (1-\eps)^{m}-
\frac{1}{1-\gamma}\limsup_{L\to\infty}
\frac{\text{tr}(T^{\eps}_{L,A_L})-\text{tr}(T^{\eps}_{L,A_L}\circ T^{\eps}_{L,A_L})}
{k_L\frac{\vol(B(\xi,R/L))}{\vol(M)}}
\end{align*}
Observe that
\begin{equation}\label{comparisondenominatorgeneral}
k_L\frac{\vol(B(\xi,R/L))}{\vol(M)}\simeq R^{m}.
\end{equation}
Applying \eqref{comparisondenominatorgeneral} and Proposition \ref{traceestimategeneral}, we have
\begin{align*}
& \liminf_{L\to\infty}\frac{\#(\Z(L)\cap A_L)+O(R^{m-1})}{k_L\frac{\vol(B(\xi,R/L))}{\vol(M)}}\\
& \geq (1-\eps)^{m}
-\frac{C}{1-\gamma}\frac{\limsup_{L\to\infty}(\text{tr}(T^{\eps}_{L,A_L})
-\text{tr}(T^{\eps}_{L,A_L}\circ T^{\eps}_{L,A_L}))}{R^m}\\
& \geq (1-\eps)^{m}-\frac{C}{1-\gamma}(1-(1-\eps)^{m})-\frac{1}{1-\gamma}\frac{O(R^{m-1})}{R^m}.
\end{align*}
Taking inferior limits when $R\to \infty$ in the last estimate, we get that
\[
D^{-}(\Z)\geq (1-\eps)^{m}-\frac{C}{1-\gamma}(1-(1-\eps)^{m})Ê\quad \forall \eps>0,
\]
where $C$ and $\gamma$ are independent of $\eps$. Therefore, letting $\eps\to 0$ we get 
the claimed result:
\[
D^{-}(\Z)\geq 1.
\]
Assume now that $\Z$ is an $L^2$-interpolating family, in particular it is uniformly separated by 
Proposition \ref{interpsep}. Fix $\rho>0$. Let $\delta>0$ be the value given by Lemma 
\ref{controlvapsconcentraciointerpolaciogeneral}.
\begin{align*}
& \#\left\{\lambda^{L(1+\rho)}_{j}\geq \delta\right\}
\leq \frac{-1}{\delta}\text{tr}(T^{\eps}_{L(1+\rho),A_L}\circ T^{\eps}_{L(1+\rho),A_L})+
\frac{1+\delta}{\delta}\text{tr}(T^{\eps}_{L(1+\rho),A_L})\\
& =\text{tr}(T^{\eps}_{L(1+\rho),A_L})+\frac{1}{\delta}(\text{tr}(T^{\eps}_{L(1+\rho),A_L})
-\text{tr}(T^{\eps}_{L(1+\rho),A_L}\circ T^{\eps}_{L(1+\rho),A_L})).
\end{align*}
Using the remark following Lemma \ref{controlvapsconcentraciointerpolaciogeneral}, 
\eqref{comparisondenominatorgeneral}, \eqref{modifiedtracequotientabove} and 
Proposition \ref{traceestimategeneral} we have
\begin{align*}
& \limsup_{L\to\infty}\frac{\#(\Z(L)\cap A_L)-O(R^{m-1})}{k_L\frac{\vol(B(\xi,R/L))}{\vol(M)}}
\leq \limsup_{L\to\infty}\frac{\text{tr}(T^{\eps}_{L(1+\rho),A_L})}{k_L\frac{\vol(B(\xi,R/L))}{\vol(M)}}\\
& +\frac{1}{\delta}\limsup_{L\to\infty}
\frac{\text{tr}(T^{\eps}_{L(1+\rho),A_L})-\text{tr}(T^{\eps}_{L(1+\rho),A_L}\circ T^{\eps}_{L(1+\rho),A_L})}
{k_L\frac{\vol(B(\xi,R/L))}{\vol(M)}}+\frac{C_1}{R^m}\\
& \leq (1+\rho)^{m}+\frac{C(1+\rho)^{m}}{\delta}(1-(1-\eps)^{m})+
\frac{1}{\delta}\frac{O(R^{m-1})}{R^m}+\frac{C_1}{R^{m}}.
\end{align*}
Taking superior limits in $R\to\infty$ in the last estimate and using the expression for $\delta$, we get
\begin{align*}
D^{+}(\Z) \leq (1+\rho)^{m}+\frac{C(1+\rho)^{m}}{\beta_{\eps}^2\left(\frac{1}{1+\rho}\right)}(1-(1-\eps)^{m}), 
\quad \forall \eps,\rho>0,
\end{align*}
where $C$ is independent of $\eps>0$ and $\rho$. Thus, taking limits in $\eps\to 0$ and then in $\rho\to 0$, 
we get the claimed result:
\[
D^{+}(\Z)\leq 1.
\]
\end{proof}

\subsection{Trace estimate}\label{tracegeneral}
In this section, we prove Proposition \ref{traceestimategeneral}. For this purpose, 
we need the following computation.
\begin{lemma}\label{estimateintegralkernel}
Let $H:[0,\infty)\to[0,1]$ be a function of class $\mathcal C^{\infty}$ with compact support in 
$[0,1]$. Let $B(\xi,R/L)$ be a ball in $M$. Then
\begin{align*}
I:=\int_{B(\xi,R/L)}\int_{M\setminus B(\xi,R/L)}&
\left|\sum^{k_L}_{i=1}H(\lambda_i/L)\phi_{i}(z)\phi_i(w)\right|^2dV(w)dV(z)\\
& \leq CR^{m-1},
\end{align*}
where $C$ is independent of $L$ and $R$.
\end{lemma}
The proof follows by using Lemma \ref{kernelestimates} and working in local coordinates.
\begin{proof}[Proof of Proposition \ref{traceestimategeneral}]
Let $A=B(\xi,R/L)$. Recall the definition of the kernels $B^{\eps}_{L}(z,w)$ and 
$\tilde{B}^{\eps}_{L}(z,w)$:
\[
B^{\eps}_{L}(z,w)=\sum^{k_L}_{i=1}\beta_{\eps}\left(\frac{\lambda_i}{L}\right)\phi_{i}(z)\phi_{i}(w),
\]
\[
\tilde{B}^{\eps}_{L}(z,w)=\sum^{k_L}_{i=1}\alpha\left(\frac{\lambda_i}{L}\right)\phi_i(z)\phi_i(w)
:=\sum^{k_L}_{i=1}\beta_{\eps}^2\left(\frac{\lambda_i}{L}\right)\phi_i(z)\phi_i(w).
\]
First, we will compute the trace of $T^{\eps}_{L,A}\circ T^{\eps}_{L,A}$.
\begin{align*}
& \tr(T^{\eps}_{L,A}\circ T^{\eps}_{L,A})=\int_{A\times A}|\tilde{B}^{\eps}_{L}(z,w)|^2dV(w)dV(z)\\
& =\int_{A}\sum^{k_L}_{i=1}\alpha^2\left(\frac{\lambda_i}{L}\right)\phi^2_i(z)dV(z)
-\int_{A}\int_{M\setminus A}|\tilde{B}^{\eps}_{L}(z,w)|^2dV(w)dV(z).
\end{align*}
Thus, we have
\begin{align*}
& \tr(T^{\eps}_{L,A})-\tr(T^{\eps}_{L,A}\circ T^{\eps}_{L,A})
=\int_{A}\sum^{k_L}_{i=1}\left[\alpha\left(\frac{\lambda_i}{L}\right)
-\alpha^2\left(\frac{\lambda_i}{L}\right)\right]\phi^2_i(z)dV(z)\\
&+\int_{A}\int_{M\setminus A}|\tilde{B}^{\eps}_{L}(z,w)|^2dV(w)dV(z)=:I_{1}+I_{2}.
\end{align*}
By Lemma \ref{estimateintegralkernel}, $I_{2}=O(R^{m-1})$ with constants independent of $L$ 
(the constant may depend on $\eps$). Now we need to estimate $I_1$. Note 
that $\alpha(x)\equiv 1$ for $0\leq 0\leq 1-\eps$. Hence,
\begin{align*}
I_1 & =\int_{A}\sum_{\lambda_i\in (L(1-\eps),L]}\left[\alpha\left(\frac{\lambda_i}{L}\right)
-\alpha^2\left(\frac{\lambda_i}{L}\right)\right]\phi^2_i(z)dV(z)\\
& \leq \int_{A}\sum_{\lambda_i\in (L(1-\eps),L]}\phi^2_{i}(z)dV(z)
=\int_{A}(K_L(z,z)-K_{L(1-\eps)}(z,z))dV(z).
\end{align*}
Using the expression of the reproducing kernel (see Section \ref{estimatesReprKernel}), we 
obtain:
\[
K_L(z,z)-K_{L(1-\eps)}(z,z)=c_{m}L^{m}(1-(1-\eps)^{m})
+O(L^{m-1})(1-(1-\eps)^{m-1}).
\]
Thus,
\begin{align*}
I_1 & \leq c_m(1-(1-\eps)^{m})L^{m}\vol(B(\xi,R/L))+
\frac{o(L^{m})}{L^{m}}(1-(1-\eps)^{m-1})\\
& \leq C(1-(1-\eps)^{m})R^{m}
+\frac{o(L^{m})}{L^{m}}(1-(1-\eps)^{m-1}),
\end{align*}
where $C$ is independent of $L$, $R$ and $\eps$. Therefore,
\[
\lim_{L\to\infty}I_{1}\leq C(1-(1-\eps)^{m})R^{m}.
\]
If $\rho>0$ then a similar computation, working with $L(1+\rho)$ instead of $L$, 
shows the second claim of Proposition \ref{traceestimategeneral}.
\end{proof}

\subsection{Technical results}\label{technicalresultsgeneral}
In this section, we present a proof of Lemma 
\ref{controlvapsconcentraciosamplinggeneral} and 
\ref{controlvapsconcentraciointerpolaciogeneral}. First, we shall prove a 
localization type property of the functions $f_L$ of the space $E_L$.

\begin{lemma}\label{construcciocassampling}
Let $\Z$ be a $s$-separated family. Given $f_L\in E_L$ 
and $\eta>0$, there exists $t_{0}=t_{0}(\eta)$ such that for all $t\geq t_{0}$,
\[
	\frac{1}{k_L}\sum_{z_{Lj}\notin A^{+}_L(t)}|f_{L}(z_{Lj})|^2\leq 
	C_{1}\int_{M\setminus A_L}|f_L|^2+C_{2}\eta\int_{A_L}|f_L|^2,
\]
where $A^{+}_L=A^{+}_L(t)=B(\xi,(R+t)/L)$, $C_{1}$ and $C_{2}$ are constants depending only on the manifold $M$ and 
the separation constant $s$ of $\Z$.
\end{lemma}

\begin{proof}
Let $f_L\in E_L$. Consider the kernel
\[
B_{2L}(z,w):=B^{1/2}_{2L}(z,w),
\]
where $B^{\eps}_{L}(z,w)$ is defined in \eqref{defnewkernel}. Note that the 
transform $B_{2L}|_{E_L}$ is the identity transform, by construction. Thus,
\begin{equation}\label{expressionfLintegral}
f_L(z)=B_{2L}(f_L)(z)=\int_{M}B_{2L}(z,w)f_{L}(w)dV(w),\quad \forall z\in M.
\end{equation}
By Lemma \ref{kernelestimates}, for any $N\geq m$, there exists a constant 
$C_{N}$ such that
\begin{equation}\label{B2Lkernelestimate}
|B_{2L}(z,w)|\leq C_{N}L^{m}\frac{1}{(1+2Ld(z,w))^{N}}.
\end{equation}
We will choose $N$ later on.\\
\\
In order to prove the claimed result, we will show that
\begin{enumerate}
	\item Given $\eta>0$ there exists $t_{0}=t_{0}(\eta)$ such that for all $t\geq t_0$,
	\begin{equation}\label{L1sampling}
		\frac{1}{k_L}\sum_{z_{Lj}\notin A^{+}_L}|f_{L}(z_{Lj})|\leq 
		C_{1}\int_{M\setminus A_L}|f_L|+C_{2}\eta\int_{A_L}|f_L|,
	\end{equation}
	where $C_{i}$ are uniform constants.
	\item Given $\eta>0$ there exists $t_{0}=t_{0}(\eta)$ such that for all $t\geq t_0$,
	\begin{equation}\label{Linfsampling}
		\max_{z_{Lj}\notin A^{+}_L}|f_{L}(z_{Lj})|\leq 
		C_{1}\|f_L\|_{L^{\infty}(M\setminus A_L)}+C_{2}\eta\|f_L\|_{L^{\infty}(A_L)},
	\end{equation}
	where $C_{i}$ are uniform constants.
\end{enumerate}

Hence, by interpolation we will have the claimed result for the $L^2$-norm. Let's prove 
first that this is true in the $L^{\infty}$-norm.\\
\\
Observe that the set of points $z_{Lj}\notin A^{+}_L$ is contained in 
$M\setminus B(\xi, (R+t)/L)$. Thus,
\[
\max_{z_{Lj}\notin A^{+}_L}|f_L(z_{Lj})|\leq 
\left\|f_L\right\|_{L^{\infty}(M\setminus B(\xi,(R+t)/L))}\leq
\|f_L\|_{L^{\infty}(M\setminus B(\xi, R/L))}.
\]
Hence, \eqref{Linfsampling} is trivially true.\\
\\
Now we just need to prove this for the $L^{1}$-norm. Let
\[
0\leq h_j(w):=\frac{1}{(1+2Ld(z_{Lj},w))^{N}}\leq 1.
\]
Using \eqref{expressionfLintegral} and \eqref{B2Lkernelestimate}, we obtain:
\begin{align*}
\frac{1}{k_L}\sum_{z_{Lj}\notin A^{+}_L}|f_{L}(z_{Lj})|
& \leq C_{N}\left\{\int_{M\setminus B(\xi,R/L)}+\int_{B(\xi,R/L)}\right\}
|f_L(w)|\sum_{z_{Lj}\notin A^{+}_L}h_{j}(w)\\
& =:I_{1}+I_{2}.
\end{align*}
Observe that for all $w\in M$,
\begin{align*}
h_{j}(w) \lesssim \frac{L^{m}}{s^{m}}\int_{B(z_{Lj},s/L)} \frac{dV(z)}{(1+2Ld(z,w))^{N}}.
\end{align*}
Note that $B(z_{Lj},s/L)$ are pairwise disjoint and for $w\in B(\xi,R/L)$,
\[
\bigcup_{z_{Lj}\notin A^{+}_L}B\left(z_{Lj},\frac{s}{L}\right)\subset 
M\setminus B\left(\xi,\frac{R+t-s}{L}\right)
\subset M\setminus B\left(w,\frac{t-s}{L}\right),
\]
Therefore, if $w\in B(\xi,R/L)$,
\begin{align*}
\sum_{z_{Lj}\notin A^{+}_L}h_{j}(w) \lesssim 
\frac{L^{m}}{s^{m}}\int_{M\setminus B\left(w,\frac{t-s}{L}\right)}
\frac{dV(z)}{(1+2Ld(z,w))^{N}}
\lesssim \frac{C_N}{s^{m}(t-s)^{N-m}}\leq \eta
\end{align*}
for all $t\geq t_{0}(\eta,N)$, provided $N>m$. This implies that
\[
I_{2}\leq C_{2}\eta\int_{B(\xi,R/L)}|f_L|.
\]
The only thing left is to bound the integral $I_{1}$. Given $w$, let
\[
\#J:=\#\left\{j: B(w,2s/L)\capÊB(z_{Lj},s/L)\neq \emptyset\right\}.
\]
Then there exists a uniform constant $C(s)$ (depending only on $s$) such that 
$\# J\leq C(s)$.  Hence,
\begin{align*}
\sum_{z_{Lj}\notin A^{+}_L}h_{j}(w)=\sum_{\overset{z_{Lj}\notin A^{+}_L}{j\in J}}
h_{j}(w)+\sum_{\overset{z_{Lj}\notin A^{+}_L}{j\notin J}}h_{j}(w)
\leq C(s)+\sum_{j\notin J}h_{j}(w).
\end{align*}
Note that for any $w\in M$,
\[
\cup_{j\notin J}B(z_{Lj},s/L)\subset M\setminus B(w,s/L).
\]
Hence,
\begin{align*}
\sum_{j\notin J}h_{j}(w)\lesssim 
\frac{L^m}{s^m}\int_{M\setminus B(w,s/L)}\frac{dV(z)}{(1+2Ld(z,w))^{N}}
\lesssim C_{s,N},
\end{align*}
provided $N>m$. So we have that
\[
I_{1}\leq (C(s)+C_{s,N})\int_{M\setminus B(\xi,R/L)}|f_L|dV
\]
and the claim is proved.
\end{proof}

\begin{lemma}\label{construcciocasinterpolacio}
Let $\Z$ be a $s$-separated family. Given $f_L\in E_L$ and $\eta>0$, 
there exists $t_{1}=t_{1}(\eta)$ such that for all $t\geq t_{1}$
\[
	\frac{1}{k_L}\sum_{z_{Lj}\in A^{-}_L(t)}|f_{L}(z_{Lj})|^2\leq 
	C_{1}\int_{A_L}|f_L|^2+C_{2}\eta\int_{M\setminus A_L}|f_L|^2,
\]
where $A^{-}_L=A^{-}_L(t)=B(\xi, (R-t)/L)$, $C_{1}$ and $C_{2}$ are constants depending 
only on the manifold $M$ and the separation constant  $s$ of $\Z$.
\end{lemma}
The proof of Lemma \ref{construcciocasinterpolacio} is similar to 
the proof of Lemma \ref{construcciocassampling}.\\
\\
Now we prove Lemma \ref{controlvapsconcentraciosamplinggeneral}.
\begin{proof}[Proof of Lemma \ref{controlvapsconcentraciosamplinggeneral}]
Given $F_{L}\in E_L$, assume that
\[
F_L(z_{Lj})=0,\quad \forall z_{Lj}\in A^{+}_{L}=B(\xi,(R+t)/L).
\]
Then, using the fact that $\Z$ is $L^2$-\textup{M-Z} and Lemma \ref{construcciocassampling}, 
we have
\begin{align*}
\|F_L\|^2_{2}& \lesssim \frac{1}{k_L}\sum^{m_L}_{j=1}|F_L(z_{Lj})|^2=
\frac{1}{k_L}\sum_{z_{Lj}\notin A^{+}_{L}}|F_{L}(z_{Lj})|^2\\
& \leq C_{1}\int_{M\setminus A_L}|F_L|^2+C_{2}\eta\int_{A_{L}}|F_L|^2
\leq C_{1}\int_{M\setminus A_L}|F_L|^2+C_{2}\eta\|F_L\|^2_{2}.
\end{align*}
Picking $\eta>0$ small enough (note that it is independent of $\eps$, $L$ and $R$), we 
get a $t_0(\eta)$ given by Lemma \ref{construcciocassampling} so that for all $t\geq t_{0}$,
\begin{equation}\label{conclusion1}
\|F_L\|^2_{2}\leq C_3 \int_{M\setminus A_L}|F_L|^2dV,
\end{equation}
where $F_L\in E_L$ is any function vanishing at the points $z_{Lj}$ that are contained 
in $A^{+}_L$. Observe that $C_{3}>1$.\\
\\
Now, we consider an orthonormal basis of eigenvectors $G^{L}_{j}$ corresponding 
to the eigenvalues $\lambda^{L}_{j}$ of the modified concentration operator. Let
\[
f_L(z)=\sum^{N_{L}+1}_{j=1}c^{L}_{j}G^{L}_{j}\in E_L.
\]
Note that $f_L\in E_L$ since $N_L\leq CR^{m}\leq k_L$ for $L$ big enough, in view of 
the separation of $\Z$. Consider now $F_L:=B^{\eps}_L(f_L)\in E_L$. We will 
apply inequality \eqref{conclusion1} to $F_L$. We pick $c^{L}_{j}$ such that 
$F_{L}(z_{Lj})=0$ for all $z_{Lj}\in A^{+}_{L}$. Observe that
\begin{align*}
\sum^{N_L+1}_{j=1}\lambda^{L}_{j}|c^{L}_{j}|^2=\langle T^{\eps}_{L,A_L}f_L,f_L\rangle
=Ê\int_{A_L}|B^{\eps}_{L}f_L(w)|^2dV(w).
\end{align*}
Now, using inequality \eqref{conclusion1},
\begin{align*}
& \lambda^{L}_{N_{L}+1} \sum^{N_L+1}_{j=1}|c^{L}_{j}|^2\leq 
\sum^{N_L+1}_{j=1}\lambda^{L}_{j}|c^{L}_{j}|^2
=\left\{\int_{M}-\int_{M\setminus A_L}\right\}|B^{\eps}_Lf_L(z)|^2dV\\
& \leq \left(1-\frac{1}{C_3}\right)\left\|B^{\eps}_L(f_L)\right\|^2_{2}
\leq \left(1-\frac{1}{C_3}\right)\|f_L\|^2_{2}
=\left(1-\frac{1}{C_3}\right)\sum^{N_L+1}_{j=1}|c^{L}_{j}|^2.
\end{align*}
where the constant $C_3$ comes from \eqref{conclusion1} (independent of 
$\eps$, $L$ and $R$). Hence,
\[
\lambda^{L}_{N_L+1}\leq 1-\frac{1}{C_3}=:\gamma<1.
\]
\end{proof}

Now we are going to prove the technical lemma corresponding to the 
interpolating case.

\begin{proof}[Proof of Lemma \ref{controlvapsconcentraciointerpolaciogeneral}]
Let $\I=\left\{j;\quad z_{Lj}\in A^{-}_L\right\}$ and $\rho>0$ fixed.\\
Recall that, by Lemma \ref{propinterpfunc}, if $\Z$ is an interpolating sequence, 
then for each sequence $\left\{c_{Lj}\right\}_{Lj}$ such that
\[
\sup_{L}\frac{1}{k_L}\sum^{m_L}_{j=1}|c_{Lj}|^2<\infty,
\]
we can construct functions $f_L\in e(L)$ with $\sup_L\|f_L\|_{2}<\infty$ and 
$f_L(z_{Lj})=c_{Lj}$, where
\[
e(L):=\left\{f_L\in E_L;\quad \left\|f_L\right\|^2_{2}\lesssim 
\frac{1}{k_L}\sum^{m_L}_{j=1}|f_L(z_{Lj})|^2\right\}.
\]
In fact, these functions $f_L$ are the solution of the interpolation problem with 
minimal norm.\\
Since we have an interpolating family, we can construct for each 
$z_{Lj}\in \Z(L)$ a function $f_{j}\in e(L)$ such that
\[
f_{j}(z_{Lj'})=\delta_{jj'}.
\]
Clearly these functions $f_{j}$ are linearly independent. Since 
$B^{\eps}_{L(1+\rho)}|_{E_L}$ is bijective, for each $j$ there exists 
a function $h_j\in E_L$ such that
\[
f_{j}=B^{\eps}_{L(1+\rho)}h_j.
\]
Let
\[
F:=\text{span}\left\{h_j;\quad z_{Lj}\in A^{-}_L\right\}.
\]
Note that $F$ has dimension $n_L$. Let $f_L\in F$ an arbitrary function 
and $g_L:=B^{\eps}_{L(1+\rho)}f_L$. Since $f_L\in F$, we know that
\[
f_L=\sum_{j\in \I}c_jh_j.
\]
Hence,
\[
g_L=B^{\eps}_{L(1+\rho)}f_L=\sum_{j\in\I}c_jB^{\eps}_{L(1+\rho)}h_j
=\sum_{j\in \I}c_jf_j\in e(L),
\]
where we have used that each $f_j\in e(L)$ and so this $g_L$ is the function 
of minimal norm that solves the interpolation problem with data $c_j\delta_{jj'}$. 
Therefore,
\[
\|g_L\|^2_2\lesssim \frac{1}{k_L}\sum^{m_L}_{j=1}|g_L(z_{Lj})|^2,
\]
where the constant do not depend on $\eps$ and $L$.\\
\\
Note that, by construction, $f_j$ vanishes in the points $z_{Lj'}$ with $j\neq j'$. 
Therefore, for each $j\in\I$ fixed, we have that $f_{j}(z_{Lk})=0$ for all $k\notinÊ\I$. 
Thus,
\[
g_L(z_{Lk})=\sum_{j\in \I}c_{j}f_{j}(z_{Lk})=0,\quad \forall k\notin\I.
\]
This shows that $g_L=0$ for $z_{Lk}\notin A^{-}_L$. Hence, applying  
Lemma \ref{construcciocasinterpolacio} to $g_L=B^{\eps}_{L(1+\rho)}f_L$, we get
\begin{align*}
& \|B^{\eps}_{L(1+\rho)}f_L\|^2_{2}\lesssim \frac{1}{k_L}\sum^{m_L}_{j=1}|g_L(z_{Lj})|^2
=\frac{1}{k_L}\sum_{j\in \I}|g_L(z_{Lj})|^2\\
& \leq C_{1}\int_{A_L}|B^{\eps}_{L(1+\rho)}f_L|^2dV
+C_{2}\eta\int_{M\setminus A_L}|B^{\eps}_{L(1+\rho)}f_L|^2dV\\
& \leq C_{1}\int_{A_L}|B^{\eps}_{L(1+\rho)}f_L|^2dV+C_{2}\eta\|B^{\eps}_{L(1+\rho)}f_L\|^2_2.
\end{align*}
Picking $\eta$ small enough (note that it is independent of $\rho$, $\eps$, $L$ and $R$ 
because all the constants appearing in the above computation are independent of these 
parameters), we get from Lemma \ref{construcciocasinterpolacio} a value $t_1=t_{1}(\eta)$ 
such that for all $t\geq t_1$,
\begin{equation}\label{ineqgL}
	\|B^{\eps}_{L(1+\rho)}f_L\|^2_{2}\leq C_{1}\int_{A_L}|B^{\eps}_{L(1+\rho)}f_L|^2dV.
\end{equation}
Thus, using this last estimate \eqref{ineqgL}, we get the following.
\begin{align*}
& \beta_{\eps}^{2}\left(\frac{1}{1+\rho}\right)\|f_L\|^2_{2}\leq
\sum^{k_L}_{i=1}\beta_{\eps}^2\left(\frac{\lambda_i}{L(1+\rho)}\right)|\langle f_L,\phi_{i}\rangle|^2\\
&=\sum^{k_{L(1+\rho)}}_{i=1}\beta_{\eps}^2\left(\frac{\lambda_i}{L(1+\rho)}\right)
|\langle f_L,\phi_{i}\rangle|^2= \|B^{\eps}_{L(1+\rho)}f_L\|^2_{2}\\
& \leq C_{1}\int_{A_L}|B^{\eps}_{L(1+\rho)}f_L|^2dV
=C_{1}\langle T^{\eps}_{L(1+\rho),A_L}f_L,f_L\rangle.
\end{align*}
We have proved that for all $f_L\in F$,
\begin{equation}\label{estimatespaceFgeneral}
\frac{\langle T^{\eps}_{L(1+\rho),A_L}f_L,f_L\rangle}{\langle f_L,f_L\rangle}\geq 
\delta:=C\beta_{\eps}^2\left(\frac{1}{1+\rho}\right),
\end{equation}
where $C$ does not depend on $L$, $\rho$, $\eps$ and $f_L$. Now, applying Weyl-Courant 
Lemma (see \cite[Part 2, p. 908]{DunSch}), we know
\[
\lambda^{L(1+\rho)}_{k-1}\geq 
\inf_{g\in E_{L(1+\rho)}\cap E}
\frac{\langle T^{\eps}_{L(1+\rho),A_L}g,g\rangle}{\langle g,g\rangle}
\]
for each subspace $E\subset E_{L(1+\rho)}$ with dim$(E)=k$. Take 
$E:=F\subset E_L\subset E_{L(1+\rho)}$ 
defined previously. Note that dim$(E)=$dim$(F)=n_L$ and hence, 
using \eqref{estimatespaceFgeneral}
\[
\lambda^{L(1+\rho)}_{n_L-1}\geq
\inf_{f_L\in F}\frac{\langle T^{\eps}_{L(1+\rho),A_L}f_L,f_L\rangle}{\langle f_L,f_L\rangle}
\geq \delta.
\]
Note that $0<\delta=C\beta_{\eps}^2(1/(1+\rho))<1$.
\end{proof}

\section{Fekete families}
Throughout this section only, we assume that $M$ is an admissible manifold (at 
the end of this section we provide some examples of such manifolds). The 
precise definition of admissibility is the following.
\begin{defini}\label{defadmissible}
We say that a manifold is \textbf{admissible} if it satisfies the following 
\textit{product property}: there exists a constant $C>0$ such that for all $0<\eps<1$ and $L\geq 1$:
\begin{equation}\label{productproperty}
E_{L}\cdot E_{\eps L}\subset E_{L(1+C\eps)}.
\end{equation}
\end{defini}
Thus, we are assuming that we may multiply two functions of our spaces and still 
obtain a function which is in some bigger space $E_L$.
\begin{defini}
Let $\left\{\phi^{L}_{1},\ldots,\phi^{L}_{k_L}\right\}$ be any basis in $E_L$. The 
points $\Z(L)=\left\{z_{L1},\ldots,z_{Lk_L}\right\}$ maximizing the determinant
\[
|\Delta(x_1,\ldots,x_{k_L})|=|\text{det}(\phi^L_{i}(x_j))_{i,j}|
\]
are called a set of \textbf{Fekete points} of \textit{degree} $L$ for $M$.
\end{defini}
A natural problem is to find the limiting distribution of points as $L\to\infty$. In 
\cite{MarOrtFekete}, J. Marzo and J. Ortega-Cerd\`a proved that as $L\to\infty$, the 
number of Fekete points of degree $L$ for $\mathbb S^m$ in a spherical cap $B(z,R)$ 
gets closer to $k_L\tilde{\sigma}(B(z,R))$, where $\tilde{\sigma}$ is the normalized 
Lebesgue measure on $\mathbb S^{m}$. They emphasize the connection of the 
Fekete points with the M-Z and interpolating arrays. In \cite{BB}, Berman and Boucksom 
have found the limiting distribution in the context of line bundles over complex manifolds. 
The proof is based on a careful study of the weighted transfinite diameter and its 
differentiability.\\
Following the approach in \cite{MarOrtFekete}, we study the distribution of a set of 
Fekete points associated to the spaces $E_L$ as $L\to\infty$. The main difficulty in 
relating the Fekete points with the M-Z and interpolating families is to construct a 
weighted interpolation formula for $E_L$ where the weight has a fast decay off the 
diagonal. That is the reason why, we restrict our attention to manifolds that satisfy 
the pro\-duct property \eqref{productproperty}. Under this hypothesis, we are able to 
prove the equidistribution of the Fekete points.

\subsection{Relation with interpolating and M-Z arrays}
The following two results give the relation of the Fekete points with the interpolating and 
M-Z arrays. Intuitively, Fekete families are almost interpolating and M-Z.

\begin{thm}\label{FeketeMZ}
Given $\eps>0$, let $L_{\eps}=[(1+\eps)L]$ and
\[
\Z_{\eps}(L)=\Z(L_\eps)=[z_{L_{\eps}1},\ldots,z_{L_{\eps}k_{L_{\eps}}}],
\]
where $\Z(L)$ is a set of Fekete points of degree $L$. Then 
$\Z_{\eps}=\left\{\Z_{\eps}(L)\right\}_{L}$ is a M-Z array.
\end{thm}

\begin{proof}
Assume that $\Z$ is a Fekete family. We will prove that they are uniformly separated. 
Consider the Lagrange \textit{polynomial} defined as
\[
l_{Li}(z):=\frac{\Delta(z_{L1},\ldots,z_{L(i-1)},z,z_{L(i+1)},\ldots,z_{Lk_L})}
{\Delta(z_{L1},\ldots,z_{Lk_L})}.
\]
Note that
\begin{itemize}
	\item $\|l_{Li}\|_{\infty}=1$.
	\item $l_{Li}(z_{Lj})=\delta_{ij}$.
	\item $l_{Li}\in E_L$.
\end{itemize}
Thus, using the Bernstein inequality for the space $E_L$ (see \eqref{BernsteinIneqforEL}), 
we have for all $j\neq i$,
\begin{align*}
1& =|l_{Li}(z_{Li})-l_{Li}(z_{Lj})|\leq \|\nabla l_{Li}\|_{\infty}d_M(z_{Li},z_{Lj})\\
& \lesssim L\|l_{Li}\|_{\infty}d_M(z_{Li},z_{Lj})=Ld_M(z_{Li},z_{Lj}).
\end{align*}
Therefore,
\[
d_M(z_{Li},z_{Lj})\geq \frac{C}{L},
\]
i.e. $\Z$ is uniformly separated. This implies that $\Z_{\eps}$ is also uniformly separated 
because
\[
d_M(z_{L_{\eps}i},z_{L_{\eps}j})\geq \frac{C}{L_{\eps}}\overset{L_{\eps}\leq (1+\eps)L}{\geq}
\frac{C/(1+\eps)}{L}.
\]
Using Theorem \ref{unionsepfam} we get for any $f_L\in E_L$,
\[
\frac{1}{k_L}\sum^{k_{L_{\eps}}}_{j=1}|f_L(z_{L_{\eps}j})|^2\lesssim \int_{M}|f_L|^2dV.
\]
In order to prove that $\Z_{\eps}$ is M-Z, we only need to prove the converse inequality, i.e.
\[
\frac{1}{k_L}\sum^{k_{L_{\eps}}}_{j=1}|f_L(z_{L_{\eps}j})|^2\gtrsim \|f_L\|^2_2.
\]
Consider the Lagrange interpolation operator defined in $\mathcal C(M)$ as
\[
\Lambda_{L}(f)(z):=\sum^{k_L}_{j=1}f(z_{Lj})l_{Lj}(z).
\]
Note that
\[
\|\Lambda_L(f)\|_{\infty}\leq k_{L}\|f\|_{\infty}.
\]
This estimate isn't enough. In order to have better control on the norms, we will make use of 
a weighted interpolation formula. Fix a point $z\in M$ and let $p(z,\cdot)$ be a function in the space 
$E_{\frac{\eps}{C}L}$ such that $p(z,z)=1$, where $C$ is the constant appearing in 
\eqref{productproperty}. Then given $f_L\in E_L$ one has
\[
R(w)=f_L(w)p(z,w)\in E_{L_{\eps}}.
\]
Note that $R(z)=f_L(z)p(z,z)=f_L(z)$. Thus, we have a weigthed representation formula
\[
f_L(z)=\sum^{k_{L_{\eps}}}_{j=1}p(z,z_{L_{\eps}j})f_L(z_{L_{\eps}j})l_{L_{\eps}j}(z).
\]
We define the operator $Q_L$ from $\mathbb C^{k_{L_{\eps}}}\to E_{L_{2\eps}}$ as
\[
Q_L[v](z)=\sum^{k_{L_{\eps}}}_{j=1}v_jp(z,z_{L_{\eps}j})l_{L_{\eps}j}(z),
\quad \forall v\in \mathbb C^{k_{L_{\eps}}}.
\]
We want to prove that
\begin{equation}\label{ineqFeketeMZ}
\int_{M}|Q_L[v](z)|^2dV(z)\lesssim \frac{1}{k_{L}}\sum^{k_{L_{\eps}}}_{j=1}|v_j|^2,
\end{equation}
with constant independent of $L$. Once we have proved this estimate, choosing 
$v_j=f_L(z_{L_{\eps}j})$ we will have
\begin{align*}
Q_L[(f_L(z_{L_{\eps}j}))_{j}](z)& 
=\sum^{k_{L_{\eps}}}_{j=1}f_L(z_{L_{\eps}j})p(z,z_{L_{\eps}j})l_{L_{\eps}j}(z)\\
& =\sum^{k_{L_{\eps}}}_{j=1}R(z_{L_{\eps}j})l_{L_{\eps}j}(z)=R(z)=f_L(z).
\end{align*}
Hence, applying the claimed inequality \eqref{ineqFeketeMZ} we will obtain
\[
\|f_L\|^2_2\lesssim \frac{1}{k_{L}}\sum^{k_{L_{\eps}}}_{j=1}|f_L(z_{L_{\eps}j})|^2,
\]
and thus $\Z_{\eps}$ is M-Z.\\
In order to prove \eqref{ineqFeketeMZ}, we need to choose the weight $p$ with care. 
We shall construct $p\in E_{L\eps/C}$ with a fast decay off the diagonal.\\
\\
Let $\delta>0$ and consider the kernels $B_L(z,w):=B^{\delta}_{L}(z,w)$ defined in 
Section \ref{estimatesReprKernel}. Let
\[
p(z,w)=\frac{B_{L\frac{\eps}{C}}(z,w)}{B_{L\frac{\eps}{C}}(z,z)}
\in E_{L\frac{\eps}{C}}.
\]
Observe that
\begin{itemize}
	\item $p(z,z)=1$.
	\item 
	\begin{align*}
	\int_{M}|p(z,w)|dV(w)& =\frac{1}{B_{L\frac{\eps}{C}}(z,z)}
	\|B_{L\frac{\eps}{C}}(z,\cdot)\|_1\\
	& \lesssim \frac{1}{k_L},
	\end{align*}
	where we have used $\|B_{L}(z,\cdot)\|_1\lesssim 1$ (see 
	\cite[Equation (2.11), Theorem 2.1]{FM-Bernstein} for a proof).
\end{itemize}
Now we are ready to prove \eqref{ineqFeketeMZ}. Note that
\begin{align*}
\int_{M}|Q_L[v](z)|dV(z)& \leq \int_{M}\sum^{k_{L_{\eps}}}_{j=1}|v_{j}||p(z,z_{L_{\eps}j})||l_{L_{\eps}j}(z)|dV(z)\\
& \leq \sum^{k_{L_{\eps}}}_{j=1}|v_{Lj}|\|p(\cdot,z_{L_{\eps}j})\|_1dV(z)
\lesssim \frac{1}{k_L}\sum^{k_{L_{\eps}}}_{j=1}|v_{Lj}|.
\end{align*}
On the other hand,
\[
|Q_L[v](z)|\leq\sup_{j}|v_j|\sum^{k_{L_{\eps}}}_{j=1}|p(z,z_{L_{\eps}j})|.
\]
Let $s$ be the separation constant of $Z_{L_{\eps}}$ and
\[
h(z,w)=\frac{1}{(1+L_{\eps}d_M(z,w))^N}\leq 1.
\]
Note that,
\[
\inf_{w\in B(z_{L_{\eps}j},s/L_{\eps})}h(z,w)\geq C_{s}h(z,z_{L_{\eps}j}).
\]
Therefore,
\begin{align*}
\sum^{k_{L_{\eps}}}_{j=1}|p(z,z_{L_{\eps}j})|&
=\frac{1}{B_{L\frac{\eps}{C}}(z_{L_{\eps}j},z_{L_{\eps}j})}
\sum^{k_{L_{\eps}}}_{j=1}|B_{L\frac{\eps}{C}}(z_{L_{\eps}j},z)|\\
& \lesssim \sum^{k_{L_{\eps}}}_{j=1}
\frac{1}{(1+L\frac{\eps}{C}d_M(z,z_{L_{\eps}j}))^N}\\
& \lesssim \frac{L^{m}_{\eps}}{s^m}
\int_{\cup^{k_{L_{\eps}}}_{j=1}B(z_{L_{\eps}j},s/L_{\eps})}h(z,w)dV(w)\\
&=\frac{L^{m}_{\eps}}{s^m}
\int_{\cup^{k_{L_{\eps}}}_{j=1}B(z_{L_{\eps}j},s/L_{\eps})\cap B(z,2s/L_{\eps})}
h(z,w)dV(w)\\
& +\frac{L^{m}_{\eps}}{s^m}
\int_{\cup^{k_{L_{\eps}}}_{j=1}B(z_{L_{\eps}j},s/L_{\eps})\cap B(z,2s/L_{\eps})^c}
h(z,w)dV(w)\\
& \leq C_{s,\eps}+C_{s}L^{m}_{\eps}\int_{M\setminus B(z,2s/L_{\eps})}h(z,w)dV(w)
\lesssim 1,
\end{align*}
where we have used that
\[
\int_{M\setminus B(z,r/L_{\eps})}h(z,w)dV(w)\lesssim \frac{1}{L^{m}_{\eps}(1+r)^{N-m}}.
\]
This computation follows by integrating $h(z,w)$ using the distribution function.\\
Hence, we have proved that
\[
\|Q_L[v]\|_{\infty}\lesssim \sup_j|v_j|.
\]
The claimed estimate \eqref{ineqFeketeMZ} follows by the Riesz-Thorin interpolation theorem.
\end{proof}
\noindent The following result relates the Fekete points with the interpolating families.
\begin{thm}\label{Feketeinterp}
Given $\eps>0$, let $L_{-\eps}=[(1-\eps)L]$ and let
\[
\Z_{-\eps}(L)=\Z(L_{-\eps})=\left\{z_{L_{-\eps}1},\ldots,z_{L_{-\eps}k_{L_{-\eps}}}\right\},
\]
where $\Z(L)$ is a set of Fekete points of \textit{degree} $L$. Then the array 
$\Z_{-\eps}=\left\{\Z_{-\eps}(L)\right\}_{L}$ is an interpolating family.
\end{thm}
\begin{proof}
Given any array of values $\left\{v_{L_{-\eps}j}\right\}^{k_{L_{-\eps}}}_{j=1}$, we consider
\[
R_{L}[v](z)=\sum^{k_{L_{-\eps}}}_{j=1}v_{L_{-\eps}j}p(z,z_{L_{-\eps}j})l_{L_{-\eps}j}(z)\in E_L,
\]
where $p(\cdot,z)\in E_{L\eps/C}$ defined in the proof of the previous Theorem. Note that
\begin{align*}
R_{L}[v](z_{L_{-\eps}k})& =\sum^{k_{L_{-\eps}}}_{j=1}v_{L_{-\eps}j}p(z_{L_{-\eps}k},z_{L_{-\eps}j})
l_{L_{-\eps}j}(z_{L_{-\eps}k})\\
& =v_{L_{-\eps}k}p(z_{L_{-\eps}k},z_{L_{-\eps}k})=v_{L_{-\eps}k}.
\end{align*}
Also, as in the proof of the previous theorem we have
\[
\sum^{k_{L_{-\eps}}}_{j=1}|p(z,z_{L_{-\eps}j})|\lesssim 1
\]
and
\[
\int_{M}|p(z,z_{L_{-\eps}j})|dV(z)\lesssim \frac{1}{k_L}.
\]
Thus, as before we have that
\[
|R_L[v](z)|\leq \sup_{j}|v_{L_{-\eps}j}|\sum^{k_{L_{-\eps}}}_{j=1}|p(z,z_{L_{-\eps}j})|
\lesssim \sup_j|v_{L_{-\eps}j}|.
\]
Hence
\[
\|R_L[v]\|_{\infty}\lesssim \sup_{j}|v_{L_{-\eps}j}|.
\]
Also,
\[
\|R_L[v]\|_1\leq \sum^{k_{L_{-\eps}}}_{j=1}|v_{L_{-\eps}j}|\|p(\cdot,z_{L_{-\eps}j})\|_1
\lesssim \frac{1}{k_L}\sum^{k_{L_{-\eps}}}_{j=1}|v_{L_{-\eps}j}|.
\]
By the Riesz-Thorin interpolation theorem we get
\[
\|R_L[v]\|^2_2\lesssim \frac{1}{k_L}\sum^{k_{L_{-\eps}}}_{j=1}|v_{L_{-\eps}j}|^2.
\]
\end{proof}

\subsection{Equidistribuition of the Fekete families}
Now we are ready to prove the equidistribuition of the Fekete points. Since the Fekete 
families are, essentially, interpolating and M-Z, we will make use of the density result, 
proved in the previous section, that gives a necessary condition for interpolation and M-Z. 
In what follows, $\sigma$ will denote the normalized volume measure, i.e. 
$d\sigma=dV/\vol(M)$. Our main result is:
\begin{thm}
Let $\Z=\left\{\Z(L)\right\}_{L\geq 1}$ be any array such that $\Z(L)$ is a set of Fekete 
points of degree $L$ and $\mu_L=\frac{1}{k_L}\sum^{k_L}_{j=1}\delta_{z_{Lj}}$. Then 
$\mu_L$ converges in the weak-$\ast$ topology to the normalized volume measure 
on $M$.
\end{thm}
\begin{proof}
We know that for any $\eps>0$ the array $\Z_{\eps}=\left\{\Z_{\eps}(L)\right\}_{L\geq 1}$ is 
M-Z, so if we use the density results (see Theorem \ref{charactdensitygeneral}), we get for any 
$\eps>0$, a large $R(\eps)$ and $L(R(\eps))$ such that for all $R\geq R(\eps)$ and 
$L\geq L(R(\eps))$ and $\xi\in M$,
\begin{equation}
\frac{\frac{1}{k_L}\#(\Z(L)\cap B(\xi, R/L))}{\sigma(B(\xi,R/L))}\geq (1-\eps).
\end{equation}
Similarly, since $\Z_{-\eps}$ is interpolating (because $\Z$ is a family of Fekete) we know 
that there exist $R(\eps)$ and $L(R(\eps))$ such that for all $R\geq R(\eps)$ and 
$L\geq L(\eps)$ and $\xi\in M$,
\begin{equation}
\frac{\frac{1}{k_L}\#(\Z(L)\cap B(\xi, R/L))}{\sigma(B(\xi,R/L))}\leq (1+\eps).
\end{equation}
Note that
\[
\mu_L(B(\xi,R/L))=\frac{1}{k_L}\#(\Z(L)\cap B(\xi,R/L)).
\]
Thus, for any $\eps>0$ there is a large $R$ such that for any $L$ big enough and $\xi\in M$,
\begin{equation}
(1-\eps)\sigma(B(\xi,r_L))\leq \mu_L(B(\xi,r_L))\leq (1+\eps)\sigma(B(\xi,r_L)),
\end{equation}
where $r_L=R/L$. Hence, we have that
\begin{equation}\label{quotientmuLandsigma}
\lim_{L\to\infty}\frac{\mu_L(B(z,r_L))}{\sigma(B(z,r_L))}=1,\quad r_L\to 0,
\end{equation}
uniformly in $z\in M$. This is enough to prove the equidistribuition of the Fekete points. 
We proceed now with the details. Let $f\in\mathcal C(M)$. We will use the notation
\[
\nu(f):=\int_{M}f(z)d\nu(z),
\]
where $\nu$ is a measure and $f_r$ will denote the mean of $f$ over a ball $B(z,r)$ 
with respect to the volume measure, i.e.
\[
f_{r}(z)=\frac{1}{\sigma(B(z,r))}\int_{B(z,r)}f(w)d\sigma(w).
\]
We want to show that $\mu_L(f)\to \sigma(f)$, when 
$L\to\infty$, for all $f\in\mathcal C(M)$.
\begin{align*}
|\mu_L(f)-\sigma(f)|& =|(\mu_L-\sigma)(f-f_{r_L})|+|(\mu_L-\sigma)(f_{r_L})|\\
& \leq (\mu_L(M)+\sigma(M))\|f-f_{r_L}\|_{\infty}+|(\mu_L-\sigma)(f_{r_L})|\\
& \leq 2\|f-f_{r_L}\|_{\infty}+|(\mu_L-\sigma)(f_{r_L})|.
\end{align*}
We will estimate the second term using \cite[Lemma 2]{Blum} that says
\begin{equation}\label{volumquocienteuclid}
\sup_{z\in M}\left|\frac{\sigma(B(z,r))}{|\mathbb B(0,cr)|}-1\right|=O(r^2),
\end{equation}
uniformly in $z\in M$, where $c$ is a constant depending only on the manifold and $|\cdot|$ denotes the Euclidean volume. Similarly, one has
\begin{equation}\label{euclidquocientvolum}
\sup_{z\in M}\left|\frac{|\mathbb B(0,cr)|}{\sigma(B(z,r))}-1\right|=O(r^2),
\end{equation}
because, by the compactness of $M$,
\begin{equation}\label{quotientboundvolumeuclid}
C_1\leq \frac{\sigma(B(z,r))}{|\mathbb B(0,cr)|}\leq C_2,
\end{equation}
thus,
\[
\left|\frac{|\mathbb B(0,cr)|}{\sigma(B(z,r))}-1\right|=
\left|\frac{1-\frac{\sigma(B(z,r))}{|\mathbb B(0,cr)|}}{\frac{\sigma(B(z,r))}{|\mathbb B(0,cr)|}}\right|
\leq \frac{Cr^2}{C_1}=O(r^2).
\]
Similarly,
\begin{equation}\label{sigmaquocientsigma}
\sup_{w,z\in M}\left|\frac{\sigma(B(w,r))}{\sigma(B(z,r))}-1\right|=O(r^2).
\end{equation}
Using Fubini, we obtain:
\begin{align*}
|(\mu_L-\sigma)(f_{r_L})|&
\leq \int_{M}|f(w)|\left|\int_{B(w,r_L)}\frac{d\mu_L(z)-d\sigma(z)}{\sigma(B(z,r_L))}\right|d\sigma(w)
\end{align*}
Now we deal with the second integral.
\begin{align*}
& \int_{B(w,r_L)}\frac{d\mu_L(z)-d\sigma(z)}{\sigma(B(z,r_L))}=
\int_{B(w,r_L)}\frac{d\mu_L(z)-d\sigma(z)}{\sigma(B(w,r_L))}
\frac{\sigma(B(w,r_L))}{\sigma(B(z,r_L))}\\
&=\int_{B(w,r_L)}\frac{d\mu_L(z)-d\sigma(z)}{\sigma(B(w,r_L))}+
\int_{B(w,r_L)}\frac{d\mu_L(z)-d\sigma(z)}{\sigma(B(w,r_L))}
\left(\frac{\sigma(B(w,r_L))}{\sigma(B(z,r_L))}-1\right)
\end{align*}
Thus,
\begin{align*}
&\left|\int_{B(w,r_L)}\frac{d\mu_L(z)-d\sigma(z)}{\sigma(B(z,r_L))}\right|\leq 
\frac{1}{\sigma(B(w,r_L))}\left|\mu_L(B(w,r_L))-\sigma(B(w,r_L))\right|\\
&+\int_{B(w,r_L)}\frac{1}{\sigma(B(w,r_L))}
\left|\frac{\sigma(B(w,r_L))}{\sigma(B(z,r_L))}-1\right|d\mu_L(z)+d\sigma(z).
\end{align*}
Hence, using \eqref{sigmaquocientsigma},
\begin{align*}
& |(\mu_L-\sigma)(f_{r_L})| \leq \sup_{w\in M}\left|\frac{\mu_L(B(w,r_L))}{\sigma(B(w,r_L))}-1\right|\|f\|_1\\
&+\sup_{z,w\in M}\left|\frac{\sigma(B(w,r_L))}{\sigma(B(z,r_L))}-1\right|\int_{M}|f(w)|
\left(\frac{\mu_L(B(w,r_L))}{\sigma(B(w,r_L))}+1\right)d\sigma(w)\\
&\leq \|f\|_1\left(\sup_{w\in M}\left|\frac{\mu_L(B(w,r_L))}{\sigma(B(w,r_L))}-1\right|+
Cr^2_L\left(\sup_{w\in M}\left|\frac{\mu_L(B(w,r_L))}{\sigma(B(w,r_L))}\right|+1\right)\right).
\end{align*}
Briefly, we have obtained
\begin{align*}
&|\mu_L(f)-\sigma(f)|\leq 2\|f-f_{r_L}\|_{\infty}\\
&+\|f\|_1\left(\sup_{w\in M}\left|\frac{\mu_L(B(w,r_L))}{\sigma(B(w,r_L))}-1\right|+
Cr^2_L\left(\sup_{w\in M}\left|\frac{\mu_L(B(w,r_L))}{\sigma(B(w,r_L))}\right|+1\right)\right)
\end{align*}
Letting $L\to\infty$ and using \eqref{quotientmuLandsigma}, we obtain the desired result:
\[
\mu_L(f)\to\sigma(f),\quad L\to\infty,\forall f\in \mathcal C(M).
\]
\end{proof}

\subsection{Examples of manifolds}

The basic examples are the compact two-point homogeneous spaces. These spaces, 
essentially are $\mathbb S^m$, the projective spaces over the field $\field=\mathbb R,
\mathbb C,\mathbb H$ and the Cayley Plane. In these spaces we can multiply two 
functions of the spaces $E_L$ and obtain another function of some bigger space 
$E_{L}$. Indeed, in the case of the Sphere, $E_L$ represents the spherical harmonics 
of degree less than $L$, usually denoted by $\Pi_L$. In such spaces, we know that
\[
\Pi_{2L}=\text{span}\Pi_L\Pi_L,
\]
(see \cite[Lemma 4.5]{MarThesis}). Moreover, in $\mathbb S^m$,
\[
\Pi_L\cdot \Pi_{\eps L}\subset \Pi_{L(1+\eps)}.
\]
Thus, the product property holds trivially in $\mathbb S^m$.\\
\\
\textbf{Projective Spaces.}\\
\\
The case of the Projective spaces is similar to the Sphere. In \cite[Sections 3.2 and 3.3]{Shat}, 
there is a description and an orthogonal decomposition of the harmonic polynomials on the 
projective spaces.\\
\\
Let $\field$ be the field of $\R$, $\C$ or $\mathbb H$. Consider the sphere $\mathbb S^{m-1}
\subset \field^{m}\approx \R^{d m}$, where $d=\textup{dim}_{\R}\field$. We 
define the projective space $\field\mathbb P^{m-1}$ over the field $\field$ (of dimension 
$m-1$) as the quotient
\[
\field\mathbb P^{m-1}=\mathbb S^{m-1}/\sim,
\]
where $x\sim y$ if and only if $y=\gamma x$ with $\gamma\in\field$ and $|\gamma|=1$. Consider 
the space of homogeneous polynomials of degree less than $L$ on the projective spaces:
\[
\textup{Pol}_{L}=\left\{p(x)|_{\mathbb S^{m-1}};\ x\in \R^{d m}, \textup{deg}(p)\leq L, 
p(\gamma x)=|\gamma|^Lp(x), \forall \gamma\in\field\right\}.
\]
It is immediate that $\textup{Pol}_L$ verify the product property \eqref{productproperty}. We 
will show that the spaces $E_L$ associated to $\field\mathbb P^{m-1}$ are identified with 
the spaces $\textup{Pol}_L$. This proves that the projective spaces are admissible. It is 
observed in \cite[Section 3.2]{Shat}, that $\textup{Pol}_{L}$ coincide with its subspace 
of harmonic polynomials of degree less than $L$:
\[
\textup{Pol}_{L}=\textup{Harm}_{L}=\left\{p\in \textup{Pol}_{L};\ \Delta_{\R^{d m}}p\equiv 0\right\}
\]
and an orthogonal decomposition holds:
\[
\textup{Harm}_{L}=\textup{Harm}(0)\oplus\textup{Harm}(2)\oplus 
\ldots \oplus\text{Harm}(2[L/2]),
\]
where $\textup{Harm}(2k)$ is the subspace of $\textup{Pol}_L$ of harmonics of degree $2k$. 
We claim that the spaces $E_L$ associated to the projective spaces are identified with the spaces 
$\textup{Harm}_{L}$. Thus, we need to show that $\textup{Harm}(2k)$ are the eigenspaces of 
$\Delta_{\field\mathbb P^{m-1}}$. For this purpose, it is sufficient to prove that its reproducing 
kernel, $f(x,y)$, is an eigenfunction because then for any $Y\in\text{Harm}(2k)$,
\begin{align*}
\Delta_{\field\mathbb P^{m-1}}Y(x)& =\Delta_{\field\mathbb P^{m-1}}\langle Y,f(x,\cdot)\rangle
=\langle Y, \Delta_{\field\mathbb P^{m-1}}f(x,\cdot)\rangle\\
&=-\lambda^2 \langle Y, f(x,\cdot)\rangle=-\lambda^2 Y(x).
\end{align*}
Let $h_{2k}$ be the dimension of $\text{Harm}(2k)$ and $(s_{ki})^{h_{2k}}_{I=1}$
be an orthonormal basis in $\text{Harm}(2k)$. Its kernel, can be expressed as the 
function
\[
f(x,y)=\sum^{h_{2k}}_{i=1}\overline{s_{ki}(x)}s_{ki}(y),\quad x,y\in \mathbb S^{m-1}.
\]
It is proved, in \cite[Section 3.3]{Shat}, that $f(x,y)$ is a function of $|\langle x,y\rangle|^2$,
\[
f(x,y)=q_{k}(|\langle x,y\rangle|^2),
\]
where $q_{k}$ is a function $[0,1]\to\C$. Moreover, in \cite[Section 3.3]{Shat}, we can find 
an explicit form of this function:
\[
\sum^{h_{2k}}_{i=1}\overline{s_{ki}(x)}s_{ki}(y)=
b^{d}_{k}P^{(\alpha,\beta)}_{k}(2|\langle x,y\rangle|^2-1)
=b^{d}_{k}P^{(\alpha,\beta)}_{k}(\cos(\rho(x,y))),
\]
where $\rho$ is the geodesic distance, $b^{d}_{k}$ is a constant of normalization and
\[
\alpha=\frac{d m-d-2}{2},\quad \beta=\frac{d-2}{2},\quad d=\text{dim}_{\R}\field.
\]
Note that, since the reproducing kernel $f(x,y)$ depends only on $|\langle x,y\rangle|^2$, 
we only need to take account of the radial part of the Laplacian, i.e.
\begin{equation}\label{radialpart}
\frac{1}{A(r)}\frac{\partial}{\partial r}\left(A(r)\frac{\partial}{\partial r}\right),
\end{equation}
where $A(r)=c'\sin^{d(m-2)}(r/\sqrt{2})\sin^{d-1}(\sqrt{2}r)$ (see \cite[p. 168]{Rag1}). 
Since we want to calculate the radial part of the Laplacian of functions of the form 
$f(\cos(\sqrt{2}r))$, we will make a change of variable $t=\cos(\sqrt{2}r)$ in \eqref{radialpart}. 
We proceed with the details taking into account these basic identities:
\[
\sin(\theta/2)=\pm\sqrt{\frac{1-\cos(\theta)}{2}},\quad \sin(\arccos(x))=\sqrt{1-x^2}.
\]
\begin{align*}
A(r) & =c'\sin^{d(m-2)}(\sqrt{2}r/2)\sin^{d-1}(\sqrt{2}r)\\
&=c'(1-\cos(\sqrt{2}r))^{\frac{d (m-2)}{2}}\sin^{d-1}(\arccos(t))\\
& =c'(1-t)^{\frac{d(m-2)}{2}}(1-t^2)^{\frac{d-1}{2}}
=c'(1-t)^{\frac{d(m-1)-1}{2}}(1+t)^{\frac{d-1}{2}}.
\end{align*}
Now the radial part of the Laplacian can be written also in the variable $t$ and 
it turns out to be:
\begin{align*}
& \frac{1}{A(r)}\frac{\partial}{\partial r}\left(A(r)\frac{\partial}{\partial r}\right)\\
&=c'(1-t)^{-\frac{d(m-1)-2}{2}}(1+t)^{-\frac{d-2}{2}}
\frac{\partial}{\partial t}\left((1-t)^{\frac{d(m-1)}{2}}
(1+t)^{\frac{d}{2}}\frac{\partial}{\partial t}\right).
\end{align*}
Thus, defining
\[
\alpha=\frac{d(m-1)-2}{2},\quad \beta=\frac{d-2}{2},
\]
we get that the radial part of the Laplacian is of the form
\[
c'(1-t)^{-\alpha}(1+t)^{-\beta}\frac{\partial}{\partial t}
\left((1-t)^{\alpha+1}(1+t)^{\beta+1}\frac{\partial}{\partial t}\right).
\]
It is well known (see \cite{Szego}) that the precise eigenfunctions 
of this operator are the Jacobi polynomials $P^{(\alpha,\beta)}(t)$ with eigenvalues 
$-k(k+\alpha+\beta+1)=-k(k+d m/2-1)$.\\
\\
Observe that since the polynomials are dense in $L^2(\field\mathbb P^{m-1})$, 
\[
L^{2}(\field\mathbb P^{m-1})=\bigoplus_{l\geq 0}\textup{Harm}(2l),
\]
For further details check \cite[Page 87]{Rag2}. Therefore, we know that all the eigenvalues of 
$\Delta_{\field\mathbb P^{m-1}}$ are of the form $-k(k+d m/2-1)$. A simple calculation 
shows that the spaces $E_L$ in the projective spaces are identified with the space 
of spherical harmonics (of the projective spaces) with degree less than $L$. More precisely,
\[
E_L=\textup{Harm}_{L^{\ast}}=\bigoplus^{[L^{\ast}/2]}_{l=0}\textup{Harm}(2l)=\textup{Pol}_{L^{\ast}},
\]
where $L^{\ast}=\sqrt{(d m/2-1)^2+4L^2}-(d m/2-1)>0$ for $L\geq (d m-2)/3$(note 
that $\frac{L^{\ast}}{2L}\to 1$, as $L\to\infty$). Therefore, $E_L$ satisfies the product property 
\eqref{productproperty} because the spaces $\textup{Pol}_{L^{\ast}}$ verify it. As a consequence, 
the projective spaces $\field \mathbb P^{m-1}$ are admissible.\\
\\
\noindent\textbf{Other examples.}\\
\\
Another example with a different nature is the Torus, represented as the unit rectangle $[0,1]\times [0,1]$ with the identification 
$(x,y)\sim (x,y+1)$ and $(x,y)\sim (x+1,y)$.\\
The eigenfunctions of the Laplacian are of the form $e^{2\pi i(mx+ny)}$ with $m,n\in\mathbb N$. 
Now we are ready to prove the product property. Let $f_1\in E_{L}$, i.e. $f_1$ is a linear 
combination of eigenvectors of eigenvalues less than $L^2$, i.e. we are taking pairs $(n,m)$ 
such that
\[
4\pi^2\left(n^2+m^2\right)\leq L^2,
\]
and let $f_2$ be a linear combination of eigenvectors of eigenvalue less than $\eps^2L^2$ 
($0<\eps<1$), i.e. we are taking pairs $(k,l)$ such that
\[
4\pi^2\left(k^2+l^2\right)\leq \eps^2L^2,
\]
We can compute the product of $f_1$ and $f_2$:
\[
f_1(x,y)f_2(x,y)=\sum_{n,m,k,l}c_{n,m}d_{k,l}e^{i2\pi((n+k)y+x(m+l))}.
\]
Thus, we have eigenvalues
\[
V^2:=4\pi^2\left((n+k)^2+(m+l)^2\right).
\]
We will estimate $V$ by computing $(a+b)^2$ and using the fact that
\[
\begin{cases}
	n,m\leq \frac{L}{2\pi}\\
	k,l\leq \frac{\eps L}{2\pi}\\
	\sqrt{1+x}\leq 1+x/2,\quad \forall x\geq 0\\
\end{cases}.
\]
Then we get that $V^2\leq L^2(1+\eps^2+4\eps)\leq L^2(1+5\eps)$. 
Hence, $V\leq L\sqrt{1+5\eps}\leq L(1+5/2\eps)$. Therefore, a Torus is admissible.\\
\\
Similar computations show that the Klein bottle is also admissible.\\
\\
\textbf{Product of admissible manifolds.}\\
\\
More examples can be constructed by taking products of manifolds that satisfy 
the pro\-duct assumption because if $f_1$ and $f_2$ are functions defined on two 
manifolds $M$ and $N$, respectively, then
\[
\Delta_{M\times N}(f_1\cdot f_2)=f_2\Delta_{M}f_1+f_1\Delta_N(f_2).
\]
More precisely, let $M$ and $N$ be admissible manifolds, i.e.
\[
E^{M}_{L}\cdot E^{M}_{\eps L}\subset E^{M}_{L(1+C_1\eps)},
\]
and
\[
E^{N}_{mL}\cdot E^{N}_{\eps L}\subset E^{N}_{L(1+C_2\eps)},
\]
where
\[
E^{M}_L=\langle \left\{\phi_i;\quad \Delta_M\phi_i=-\lambda^2_i\phi_i, \lambda_i\leq L\right\}\rangle,
\]
and
\[
E^{N}_L=\langle \left\{\psi_i;\quad \Delta_N\psi_i=-\mu^2_i\psi_i, \mu_i\leq L\right\}\rangle.
\]
Thus, if we consider the product manifold $M\times N$, then
\[
E^{M\times N}_L=\langle \left\{\phi_i\psi_j;\quad \lambda^2_i+\mu^2_j\leq L^2\right\}\rangle,
\]
It is a straightforward computation that $M\times N$ satisfies the condition of admissibility:
\[
E^{M\times N}_{L}\cdot E^{M\times N}_{\eps L}\subset E^{M\times N}_{L(1+C\eps)}
\]
with $C=2\max(C_1,C_2)$.
\begin{remark}
Note that the example of the torus can be reduced to this later case because it is the product of 
two $\mathbb S^1$.
\end{remark}

\def\cprime{$'$}
\providecommand{\bysame}{\leavevmode\hbox to3em{\hrulefill}\thinspace}
\providecommand{\MR}{\relax\ifhmode\unskip\space\fi MR }
\providecommand{\MRhref}[2]{%
  \href{http://www.ams.org/mathscinet-getitem?mr=#1}{#2}
}
\providecommand{\href}[2]{#2}

\end{document}